\documentclass[reqno]{amsart}
\usepackage{amsmath}
\usepackage{amsfonts}
\usepackage{amssymb}
\usepackage[dvips]{graphicx}
\usepackage{epsfig}
\usepackage{color}
\newcommand{\esssup}{{\rm ess} \sup}
\newcommand{\essinf}{{\rm ess} \inf}
\newtheorem{theorem}{Theorem}[section]
\newtheorem{remark}{Remark}[section]
\newtheorem{lemma}[theorem]{Lemma}
\newtheorem{definition}{Definition}[section]

\newtheorem{proposition}[theorem]{Proposition}

\numberwithin{equation}{section}
\begin{document}

\title[the compressible Euler equations]
	{Decay of solutions of isentropic gas dynamics for large data}
\author{Naoki Tsuge}
\address{Department of Mathematics Education, 
Faculty of Education, Gifu University, 1-1 Yanagido, Gifu
Gifu 501-1193 Japan.}
\email{tuge@gifu-u.ac.jp}
\thanks{
N. Tsuge's research is partially supported by Grant-in-Aid for Scientific 
Research (C) 17K05315, Japan.
}
\date{}

\maketitle

\begin{abstract}
In this paper, we are concerned with the Cauchy problem for 
isentropic gas dynamics. Through the contribution of many 
researchers such as Lax, P. D., Glimm, J., DiPerna, R. J. and Liu, T. P., 
the decay of solutions was established. They treated with initial 
data with the small total variation. On the other hand, the decay for large initial data 
has been open for half a century. 
Our goal is to provide a new method to analyze this problem. 
We prove the existence of a global attractor, which 
yields a decay of solutions for large data. To construct approximate solutions, we 
introduce a modified Godunov scheme.
\end{abstract}


\section{Introduction}
The present paper is concerned with isentropic gas dynamics
\begin{align}
\begin{cases}
\displaystyle{\rho_t+m_x=0,}\\
\displaystyle{m_t+\left(\frac{m^2}{\rho}+p(\rho)\right)_x
	=0,}
\end{cases}x\in{\bf R},\quad t\in{\bf R}_+,
\label{Euler}
\end{align}
where $\rho$, $m$ and $p$ are the density, the momentum and the 
pressure of the gas, respectively. If $\rho>0$, 
$v=m/\rho$ represents the velocity of the gas. For a barotropic gas, 
$p(\rho)=\rho^\gamma/\gamma$, where $\gamma\in(1,5/3]$ is the 
adiabatic exponent for usual gases.

We consider the initial value problem (\ref{Euler}) 
with the initial  data
\begin{align}  
(\rho,m)|_{t=0}=(\rho_0(x),m_0(x)),
\label{I.C.}
\end{align}
where 
\begin{align}
(\rho_0(x),m_0(x))=(\bar{\rho},0)\text{ outside a finite interval} 
\label{initial bar rho}
\end{align}
and 
$\bar{\rho}$ is a positive constant. 
The above problem \eqref{Euler}--\eqref{I.C.} can be written in the  following form 
\begin{align}\left\{\begin{array}{lll}
u_t+f(u)_x=0,\quad{x}\in{\bf R},\quad t\in{\bf R}_+,\\
u|_{t=0}=u_0(x),
\label{IP}
\end{array}\right.
\end{align}
by using  $u={}^t(\rho,m)$, $\displaystyle f(u)={}^t\!\left(m, \frac{m^2}{\rho}+p(\rho)\right)$.

We recollect the known results of the above problem. {DiPerna} \cite{D2}
 proved the global existence of solutions to   
 \eqref{Euler}--\eqref{I.C.} by the vanishing viscosity method and a compensated compactness argument.  
{DiPerna} first applied the method  to \eqref{Euler} for 
the special case where $\gamma=1+2/n$ and $n$ is an odd
integer. We notice that this result can treat with the arbitrary $L^{\infty}$ data. Subsequently, {Ding}, {Chen} and {Luo} \cite{DC1} 
and {Chen} \cite{C1} and \cite{C2} extended his analysis to any $\gamma$ in $(1,5/3]$.

On the other hand, the existence of solutions to conservation laws including \eqref{Euler} was established by Glimm \cite{G}. Glimm treatd the Cauchy problem 
with initial data having small total variation. 
The theory of decay for genuinely nonlinear 2$\times$2 systems of conservation laws was constructed by Glimm-Lax \cite{GL}. Glimm-Lax showed that if initial data are constant outside a finite interval and have locally bounded total variation and small oscillation, then the tonal 
variation of the solution of \cite{G} decays to zero at the rate $t^{-1/2}$. 
The Glimm-Lax theory had been further developed by DiPerna and Liu: 
general conservation laws with a convex entropy function \cite{D1}, 
general conservation laws with small initial data in total variation \cite{L}. However, the decay for large initial data has been open for half a century. Our goal in the present paper is to provide a new method to analyze this problem and investigate the decay structure of \eqref{Euler}. We first introduce a modified Godunov scheme
to construct approximate solutions. We next prove the existence of a global attractor, which 
yields decay estimates of solutions for large initial data.

To state our main theorem, we define the Riemann invariants $w,z$, which play important roles
in this paper, as
\begin{definition}
	\begin{align*}
	w(u):=\frac{m}{\rho}+\frac{\rho^{\theta}}{\theta}=v+\frac{\rho^{\theta}}{\theta},
	\quad{z}(u):=\frac{m}{\rho}-\frac{\rho^{\theta}}{\theta}
	=v-\frac{\rho^{\theta}}{\theta}\quad
	\left(\theta=\frac{\gamma-1}{2}\right).
	\end{align*}
\end{definition}
These Riemann invariants satisfy the following.
\begin{remark}\label{rem:Riemann-invariant}
	\begin{align}
	&|w|\geq|z|,\;w\geq0,\;\mbox{\rm when}\;v\geq0.\quad
	|w|\leq|z|,\;z\leq0,\;\mbox{\rm when}\;v\leq0.
    \nonumber\\
	&v=\frac{w+z}2,
	\;\rho=\left(\frac{\theta(w-z)}2\right)^{1/\theta},\;m=\rho v.
\label{relation-Riemann}
	\end{align}From the above, the lower bound of $z$ and the upper bound of $w$ yield the bound of $\rho$ and $|v|$.
\end{remark}

We next introduce the mechanical energy as $\displaystyle 
\eta_{\ast}(u)=\frac12\frac{m^2}{\rho}+\frac1{\gamma(\gamma-1)}\rho^{\gamma}$ and set 
\begin{align}
J(u)=\eta_{\ast}(u)
-\frac{\left(\bar{\rho}\right)^{\gamma-1}}{\gamma-1}\rho
+\dfrac{\left(\bar{\rho}\right)^{\gamma}}{\gamma}.
\label{eqn:J}
\end{align}

\begin{remark}From the convexity of $\eta_{\ast}$, we have 
\begin{align}\begin{split}
J(u)=\eta_{\ast}(u)-\eta_{\ast}(\bar{\rho},0)
-\dfrac{\partial\eta_{\ast}}{\partial \rho}(\bar{\rho},0)(\rho-\bar{\rho})-\dfrac{\partial\eta_{\ast}}{\partial m}
(\bar{\rho},0)m
\geq0.
\end{split}
\end{align}
From the conservation of mass and the energy inequality, 
we have 
\begin{align}\begin{split}
0\leq\int^x_{-\infty}J(u(y,t))dy\leq\int^{\infty}_{-\infty}J(u(x,t))dx
\leq\int^{\infty}_{-\infty}J(u_0(x))dx.
\end{split}
\label{eta}
\end{align}
\end{remark}Moreover, we define the entropy weak solution.
\begin{definition}
A measurable function $u(x,t)$ is called a global {\it entropy weak solution} of the 
Cauchy problems \eqref{IP} if 
\begin{align*}
	\int^{\infty}_{-\infty}\int^{\infty}_0u\phi_t+f(u)\phi_x dxdt
	+\int^{\infty}_{-\infty}u_0(x)\phi(x,0)dx=0
\end{align*}
holds for any test function $\phi\in C^1_0({\bf R}\times{\bf R}_+)$ and 
\begin{align*}
	\int^{\infty}_{-\infty}\int^{\infty}_0\hspace{-1ex}&\eta(u)\psi_t+q(u)\psi_x+\int^{\infty}_{-\infty}\hspace{-1ex}\eta(u_0(x))\psi(x,0)dx
	\geq0
\end{align*}
holds for any non-negative test function $\psi\in C^1_0({\bf R}\times{\bf R}_+)$, where 
$(\eta,q)$ is a pair of convex entropy--entropy flux of \eqref{Euler}.
\end{definition}
Finally, we define $\tilde{z},\tilde{w}$ by
\begin{align}
	\begin{split}
		\tilde{z}(x,t)=z(x,t)-\int^x_{-\infty}J(u(y,t))dy,\;\;
\tilde{w}(x,t)=w(x,t)-\int^x_{-\infty}J(u(y,t))dy.
	\end{split}
	\label{transformation}
\end{align}

Then, our main theorem is as follows.
\begin{theorem}\label{thm:main}
We assume that 
\begin{align}
\begin{split}
\rho_0(x)\geq0\quad{a.e.\ }x\in{\bf R},\quad \rho_0\in L^{\infty}({\bf R}),\quad\dfrac{m_0}{\rho_0}\in L^{\infty}({\bf R}).
\end{split}
\label{maintheorem1}
\end{align}

Then, there exists a global entropy weak solution of the 
Cauchy problems \eqref{IP}. Moreover, for any positive constant $\varepsilon$, 
there exist positive constants $t_0$ such that the solution satisfies 
\begin{align}
\begin{split}
&-\dfrac{\left(\bar{\rho}\right)^{\theta}}{\theta}-E_0-\varepsilon\leq \tilde{z}(x,t),\quad \tilde{w}(x,t)\leq \dfrac{\left(\bar{\rho}\right)^{\theta}}{\theta}+\varepsilon,\quad \rho(x,t)\geq0,\\
&{a.e.\ }(x,t)\in{\bf R}\times[t_0,\infty),
\end{split}
\label{maintheorem2}
\end{align}where 
\begin{align}
E_0=\int^{\infty}_{-\infty}J(u_0(x))dx.
\label{E_0}
\end{align} 
\end{theorem}

\begin{remark}We remark some points for the above theorem.

In view of \eqref{initial bar rho}, we find that 
$
-\dfrac{\left(\bar{\rho}\right)^{\theta}}{\theta}-E_0\geq\essinf_x\left(\tilde{z}(x,0)\right),\;
\esssup_x\left(\tilde{w}(x,0)\right)\geq\dfrac{\left(\bar{\rho}\right)^{\theta}}{\theta}
$. Therefore, \eqref{maintheorem2} means the decay estimate of $\essinf_x\left(\tilde{z}(x,t)\right)$ and 
$\essinf_x\left(\tilde{w}(x,t)\right)$.

We similarly obtain
$
-\dfrac{\left(\bar{\rho}\right)^{\theta}}{\theta}\geq\essinf_x\left(z(x,0)\right),\;
\esssup_x\left(w(x,0)\right)\geq\dfrac{\left(\bar{\rho}\right)^{\theta}}{\theta}
$. If $-\dfrac{\left(\bar{\rho}\right)^{\theta}}{\theta}-E_0-\varepsilon>\essinf_x\left(z(x,0)\right)$ and 
$\esssup_x\left(w(x,0)\right)>\dfrac{\left(\bar{\rho}\right)^{\theta}}{\theta}+E_0+\varepsilon$, 
\eqref{maintheorem2} yields the decay estimate of 
$\essinf_x\left(z(x,t)\right)$ and $\esssup_x\left(w(x,t)\right)$. In fact, we find that 
	\begin{align*}
\essinf_x&\left(z(x,t_0)\right)-\essinf_x\left(z(x,0)\right)\\&
={\rm ess} \inf_x\left(\tilde{z}(x,t_0)+\int^x_{-\infty}J(u(y,t))dy\right)
-\essinf_x\left(z(x,0)\right)
\\&\geq-\dfrac{\left(\bar{\rho}\right)^{\theta}}{\theta}-E_0-\varepsilon-\essinf_x\left(z(x,0)\right)>0,\\
{\rm ess} \sup_x&\left(w(x,t_0)\right)-{\rm ess} \sup_x\left(w(x,0)\right)\\&
={\rm ess} \sup_x\left(\tilde{w}(x,t_0)+\int^x_{-\infty}J(u(y,t))dy\right)-{\rm ess} \sup_x\left(w(x,0)\right)
\\&\leq\dfrac{\left(\bar{\rho}\right)^{\theta}}{\theta}+\varepsilon+E_0-{\rm ess} \sup_x\left(w(x,0)\right)<0.
	\end{align*}
	
\end{remark}

\subsection{Outline of the proof (formal argument)}

The proof of main theorem is a little complicated. Therefore, 
before proceeding to the subject, let us grasp the point of the main estimate by a formal argument. 
Although \eqref{Euler} has a discontinuous solution in general, 
we assume that solutions are smooth and the density is nonnegative in this section.

We consider the physical region $\rho\geq0$ (i.e., $w\geq z$.). Recalling Remark \ref{rem:Riemann-invariant}, it suffices to 
derive the lower bound of $z(u)$ and the upper bound of $w(u)$ to obtain the bound of $u$. To do this, we diagonalize \eqref{Euler}. 
If solutions are smooth, we deduce from \eqref{Euler} 
\begin{align}
z_t+\lambda_1z_x=0,\quad
w_t+\lambda_2w_x=0,
\label{force2}
\end{align} 
where $\lambda_1$ and $\lambda_2$ are the characteristic speeds defined as follows 
\begin{align}
\lambda_1=v-\rho^{\theta},\quad\lambda_2=v+\rho^{\theta}.
\label{char}
\end{align}

From the conservation of mass 
$\eqref{Euler}_1$ and the conservation of energy 
$\left(\eta_{\ast}\right)_t+\left(q_{\ast}\right)_x=0$, we obtain
\begin{align}
	\tilde{z}_t+\lambda_1 \tilde{z}_x=g_1(u),\quad
	\tilde{w}_t+\lambda_2 \tilde{w}_x=g_2(u),
	\label{Riemann1}
\end{align}
where 
\begin{align*}
	q_{\ast}(u)=m\left(\frac12\frac{m^2}{\rho^2}+\frac{\rho^{\gamma-1}}{\gamma-1}\right) \end{align*}
and
\begin{align}
	\begin{alignedat}{2}
		&g_1(u)=&&-\dfrac{\left(\bar{\rho}\right)^{\gamma}}{\gamma}\lambda_1+\dfrac{1}{\gamma(\gamma-1)}\rho^{\gamma+\theta}
		+\dfrac{1}{\gamma}\rho^{\gamma}v+\dfrac{1}{2}\rho^{\theta+1}v^2-\frac{\left(\bar{\rho}\right)^{\gamma-1}}{\gamma-1}\rho^{\theta+1},\\
		&g_2(u)=&&-\dfrac{\left(\bar{\rho}\right)^{\gamma}}{\gamma}\lambda_2-\dfrac{1}{\gamma(\gamma-1)}\rho^{\gamma+\theta}
		+\dfrac{1}{\gamma}\rho^{\gamma}v-\dfrac{1}{2}\rho^{\theta+1}v^2+\frac{\left(\bar{\rho}\right)^{\gamma-1}}{\gamma-1}\rho^{\theta+1}.
	\end{alignedat}
\label{inhomo2}
\end{align}

To prove Theorem \ref{thm:main}, we prepare the following proposition.
\begin{proposition}\label{proposition}
\begin{align}&	
	\begin{cases}
	\bullet\;&g_1(u(x,t))>0\text{, when }\tilde{z}(x,t)<-\dfrac{\left(\bar{\rho}\right)^{\theta}}{\theta}-E_0,\\
	\bullet\;&	g_1(u(x,t))=0\text{, when }\tilde{z}(x,t)=-\dfrac{\left(\bar{\rho}\right)^{\theta}}{\theta}-E_0,\;\int^x_{-\infty}J(u(y,t))dy=E_0\\&\text{ and }
		\rho(x,t)=\bar{\rho},
	\end{cases}\label{g1}\\	
	&
	\begin{cases}
	\bullet\;&g_2(u(x,t))<0\text{, when }\tilde{w}(x,t)>\dfrac{\left(\bar{\rho}\right)^{\theta}}{\theta},\\
		\bullet\;&g_2(u(x,t))=0\text{, when }\tilde{w}(x,t)=\dfrac{\left(\bar{\rho}\right)^{\theta}}{\theta},\;\int^x_{-\infty}J(u(y,t))dy=0\\&\text{ and }
		\rho(x,t)=\bar{\rho}.
	\end{cases}
	\label{g2}
\end{align}	
\end{proposition}
\begin{proof}

We first investigate \eqref{g1}. When $\tilde{z}(x,t)\leq-\dfrac{\left(\bar{\rho}\right)^{\theta}}{\theta}-E_0$, from 
\eqref{eta}, 
we observe that 
\begin{align}
z(x,t)=\tilde{z}(x,t)+\int^x_{-\infty}J(u(y,t))dy\leq-\dfrac{\left(\bar{\rho}\right)^{\theta}}{\theta}.
\end{align}
Then, we deduce that
\begin{align}
\begin{alignedat}{3}
&g_1(u)&=&\dfrac{1}{\gamma(\gamma-1)}\rho^{\gamma+\theta}
		+\dfrac{1}{\gamma}\rho^{\gamma}v+\dfrac{1}{2}\rho^{\theta+1}v^2
		-\dfrac{\left(\bar{\rho}\right)^{\gamma}}{\gamma}\lambda_1
		-\frac{\left(\bar{\rho}\right)^{\gamma-1}}{\gamma-1}\rho^{\theta+1}
\\
&&=&\dfrac{1}{\gamma(\gamma-1)}\rho^{\gamma+\theta}
+\dfrac{1}{\gamma}\rho^{\gamma}v+\dfrac{1}{2}\rho^{\theta+1}v^2
-\dfrac{\left(\bar{\rho}\right)^{\gamma}}{\gamma}
\left(z+\frac{3-\gamma}{\gamma-1}\rho^{\theta}\right)
-\frac{\left(\bar{\rho}\right)^{\gamma-1}}{\gamma-1}\rho^{\theta+1}
\\
&&=&\dfrac{1}{\gamma(\gamma-1)}\rho^{\gamma+\theta}
+\dfrac{1}{\gamma}\rho^{\gamma}v+\dfrac{1}{2}\rho^{\theta+1}v^2
-\dfrac{\left(\bar{\rho}\right)^{\gamma}}{\gamma}
\left(z+\frac{3-\gamma}{\gamma-1}\rho^{\theta}\right)
-\frac{\left(\bar{\rho}\right)^{\gamma-1}}{\gamma-1}\rho^{\theta+1}
\\
&&=&\dfrac{1}{\gamma(\gamma-1)}\rho^{\gamma+\theta}
+\dfrac{1}{\gamma}\rho^{\gamma}\left(z
+\dfrac{\rho^{\theta}}{\theta}\right)
+\dfrac{1}{2}\rho^{\theta+1}\left(z
+\dfrac{\rho^{\theta}}{\theta}\right)^2\\
&&&
-\dfrac{\left(\bar{\rho}\right)^{\gamma}}{\gamma}
\left(z+\frac{3-\gamma}{\gamma-1}\rho^{\theta}\right)
-\frac{\left(\bar{\rho}\right)^{\gamma-1}}{\gamma-1}\rho^{\theta+1}
\\
&&=&
\dfrac{\rho^{\theta+1}}{2}\left(z-
\dfrac{1}{\gamma}
\dfrac{\left(\bar{\rho}\right)^{\gamma}}{\rho^{\theta+1}}
+\dfrac{3\gamma-1}{\gamma(\gamma-1)}\rho^{\theta}
\right)^2+\dfrac{\gamma+1}{2\gamma^2(\gamma-1)}\rho^{\gamma+\theta}
\\&&&
-\dfrac{1}{\gamma-1}\left(\bar{\rho}\right)^{\gamma-1}
\rho^{\theta+1}+
\dfrac{\gamma+1}{\gamma^2}\left(\bar{\rho}\right)^{\gamma}
\rho^{\theta}
-\dfrac{1}{2\gamma^2}\left(\bar{\rho}\right)^{2\gamma}\dfrac{1}{\rho^{\theta+1}}.
\end{alignedat}
		\label{estimate1}
	\end{align}
	
When $\rho\leq \bar{\rho}$, since \eqref{estimate1} attains the 
minimum at $z=-\dfrac{\left(\bar{\rho}\right)^{\theta}}{\theta}$, we deduce from Appendix \ref{app:formal}
\begin{align}
g_1(u)\geq&	\dfrac{5\gamma-3}{\gamma(\gamma-1)^2}\rho^{\gamma+\theta}
-\dfrac{2(3\gamma-1)}{\gamma(\gamma-1)^2}\left(\bar{\rho}\right)^{\theta}
\rho^{\gamma}+
\dfrac{3-\gamma}{(\gamma-1)^2}\left(\bar{\rho}\right)^{\gamma-1}
\rho^{\theta+1}
-\dfrac{3-\gamma}{\gamma(\gamma-1)}\left(\bar{\rho}\right)^{\gamma}
\rho^{\theta}\nonumber\\\nonumber&+
\dfrac{2}{\gamma(\gamma-1)}
\left(\bar{\rho}\right)^{\gamma+\theta}\\
\geq&0,
\label{estimate3}
\end{align}
where the equal sign of the second inequality can be used only 
when $\rho=\bar{\rho}$.

When $\rho> \bar{\rho}$, since \eqref{estimate1} attains the 
minimum at $z=\dfrac{1}{\gamma}
\dfrac{\left(\bar{\rho}\right)^{\gamma}}{\rho^{\theta+1}}
-\dfrac{3\gamma-1}{\gamma(\gamma-1)}\rho^{\theta}$, 
we deduce from Appendix \ref{app:formal}  
\begin{align}
	g_1(u)\geq&\dfrac{\gamma+1}{2\gamma^2(\gamma-1)}\rho^{\gamma+\theta}
	-\dfrac{1}{\gamma-1}\left(\bar{\rho}\right)^{\gamma-1}
	\rho^{\theta+1}+
	\dfrac{\gamma+1}{\gamma^2}\left(\bar{\rho}\right)^{\gamma}
	\rho^{\theta}
	-\dfrac{1}{2\gamma^2}\left(\bar{\rho}\right)^{2\gamma}\dfrac{1}{\rho^{\theta+1}}\nonumber\\
\geq&0,
\label{estimate4}
\end{align}
where the equal sign of the second inequality can be used only 
when $\rho=\bar{\rho}$.

We next investigate \eqref{g2}. When $\tilde{w}(x,t)\geq\dfrac{\left(\bar{\rho}\right)^{\theta}}{\theta}$, from 
\eqref{eta}, 
we observe that 
\begin{align}
	w(x,t)=\tilde{w}(x,t)+\int^x_{-\infty}J(u(y,t))dy\geq\dfrac{\left(\bar{\rho}\right)^{\theta}}{\theta}.
\end{align}
Then, we deduce that
\begin{align*}
	\begin{alignedat}{2}
g_2(u)=&-\dfrac{\left(\bar{\rho}\right)^{\gamma}}{\gamma}\lambda_2-\dfrac{1}{\gamma(\gamma-1)}\rho^{\gamma+\theta}
		+\dfrac{1}{\gamma}\rho^{\gamma}v-\dfrac{1}{2}\rho^{\theta+1}v^2+\frac{\left(\bar{\rho}\right)^{\gamma-1}}{\gamma-1}\rho^{\theta+1}\\
		=&-\dfrac{\left(\bar{\rho}\right)^{\gamma}}{\gamma}
\left(w-\frac{3-\gamma}{\gamma-1}\rho^{\theta}\right)-\dfrac{1}{\gamma(\gamma-1)}\rho^{\gamma+\theta}
+\dfrac{1}{\gamma}\rho^{\gamma}v-\dfrac{1}{2}\rho^{\theta+1}v^2+\frac{\left(\bar{\rho}\right)^{\gamma-1}}{\gamma-1}\rho^{\theta+1}\\
=&-\dfrac{\left(\bar{\rho}\right)^{\gamma}}{\gamma}
\left(w-\frac{3-\gamma}{\gamma-1}\rho^{\theta}\right)-\dfrac{1}{\gamma(\gamma-1)}\rho^{\gamma+\theta}
+\dfrac{1}{\gamma}\rho^{\gamma}\left(w-\dfrac{\rho^{\theta}}{\theta}\right)
\\&-\dfrac{1}{2}\rho^{\theta+1}\left(w-\dfrac{\rho^{\theta}}{\theta}\right)^2+\frac{\left(\bar{\rho}\right)^{\gamma-1}}{\gamma-1}\rho^{\theta+1}\\
=&
-\dfrac{\rho^{\theta+1}}{2}\left(w+
\dfrac{1}{\gamma}
\dfrac{\left(\bar{\rho}\right)^{\gamma}}{\rho^{\theta+1}}
-\dfrac{3\gamma-1}{\gamma(\gamma-1)}\rho^{\theta}
\right)^2-\dfrac{\gamma+1}{2\gamma^2(\gamma-1)}\rho^{\gamma+\theta}
\\&
+\dfrac{1}{\gamma-1}\left(\bar{\rho}\right)^{\gamma-1}
\rho^{\theta+1}-
\dfrac{\gamma+1}{\gamma^2}\left(\bar{\rho}\right)^{\gamma}
\rho^{\theta}
+\dfrac{1}{2\gamma^2}\left(\bar{\rho}\right)^{2\gamma}\dfrac{1}{\rho^{\theta+1}}.
\end{alignedat}
\end{align*}
From \eqref{estimate3}--\eqref{estimate4}, we conclude that 
\begin{align}
	g_2(u)
	\leq0,
\end{align}
where the equal sign can be used only 
when $\rho=\bar{\rho}$.
	\end{proof}

{\bf Proof of Theorem \eqref{maintheorem2}}

From \eqref{g1}--\eqref{g2}, for any positive constant 
$\varepsilon$, there exists a positive constant $\delta$ such that 
\begin{align}
	\begin{split}
	&g_1(u)>2\delta,\text{ when }\tilde{z}\leq -\dfrac{\left(\bar{\rho}\right)^{\theta}}{\theta}-E_0-\dfrac{\varepsilon}{2}
	\text{ and }0\leq\rho,\\
	&g_2(u)<-2\delta,\text{ when }\dfrac{\left(\bar{\rho}\right)^{\theta}}{\theta}+\dfrac{\varepsilon}{2}\leq\tilde{w}\text{ and }0\leq\rho
	.
	\end{split}
\label{delta}
\end{align}

We introduce $\hat{z},\hat{w}$ as follows.
\begin{align}
\hat{z}(x,t)=\tilde{z}(x,t)-\delta t,\quad\hat{w}(x,t)=\tilde{w}(x,t)+\delta t.
	\label{transformation2}
\end{align}

We deduce from \eqref{Riemann1} that
\begin{align}
	\hat{z}_t+\lambda_1 \hat{z}_x=g_1(u)-\delta,\quad
	\hat{w}_t+\lambda_2 \hat{w}_x=g_2(u)+\delta.
	\label{Riemann2}
\end{align}
We set 
\begin{align}
M_0=\max\left\{\dfrac{\left(\bar{\rho}\right)^{\theta}}{\theta},\;-\essinf_x\left(\tilde{z}(x,0)\right)+E_0
,\;\essinf_x\left(\tilde{w}(x,0)\right)\right\}.
\label{M_0}
\end{align}
Then, we notice that
\begin{align*}
	-M_0-E_0\leq \hat{z}(x,0),\quad \hat{w}(x,0)\leq M_0.
\end{align*}

Let us prove that 
\begin{align*}
	\hat{S}_{inv}=\{(\hat{z},\hat{w})\in{\bf R}^2;{\rho}\geq0,\;\hat{z}\geq-M_0-E_0,\;\hat{w}\leq M_0\}
\end{align*}
is an invariant region for the Cauchy problem of \eqref{Riemann2} on 
\begin{align*}
0\leq t\leq \max\left(0,\dfrac{M_0-\frac{\left(\bar{\rho}\right)^{\theta}}{\theta}-\varepsilon}{\delta}\right)=:t_0.
\end{align*}
We notice that this yields \eqref{maintheorem2} on $0\leq t\leq t_0$.

To achieve this, assuming
\begin{align*}
	-M_0-E_0< \hat{z}(x,0),\quad \hat{w}(x,0)< M_0
\end{align*}
and there exist $x_{\ast}\in{\bf R},\;0<t_{\ast}\leq t_0$ such that 
the following \eqref{invariant1} or \eqref{invariant2} holds,
\begin{align}
	\begin{alignedat}{2}
		&-M_0-E_0<\hat{z}(x,t),\;\hat{w}(x,t)< M_0,\quad x\in{\bf R},\;
		0\leq t<t_{\ast}\\&\text{\hspace*{0ex}and}\quad
		\hat{z}(x_{\ast},t_{\ast})=-M_0-E_0,\;\hat{w}(x_{\ast},t_{\ast})\leq M_0,	
	\end{alignedat}\label{invariant1}\\	
	\begin{alignedat}{2}
		&-M_0-E_0<\hat{z}(x,t),\;\hat{w}(x,t)< M_0,\quad x\in{\bf R},\;
		0\leq t<t_{\ast}\\&\text{\hspace*{0ex}and}\quad
		\hat{z}(x_{\ast},t_{\ast})\geq-M_0-E_0,\;\hat{w}(x_{\ast},t_{\ast})=M_0,
	\end{alignedat}\label{invariant2}
\end{align}
we will deduce a contradiction. 

To do this, we prove 
\begin{align}	&g_1(u(x_{\ast},t_{\ast}))-\delta>0\text{, when \eqref{invariant1} holds},
	\label{g1check}\\
	&g_2(u(x_{\ast},t_{\ast}))+\delta<0\text{, when \eqref{invariant2} holds}.
	\label{g2check}
\end{align} 

Let us consider \eqref{g1check}. When \eqref{invariant1} and $0\leq t_{\ast}\leq t_0$, 
we notice that
\begin{align*}
\tilde{z}(x_{\ast},t_{\ast})\leq-\frac{\left(\bar{\rho}\right)^{\theta}}{\theta}-E_0-\varepsilon.
\end{align*}
Therefore, from \eqref{delta}, we prove \eqref{g1check}. Since $\hat{z}$ attains the minimum at $(x_{\ast},t_{\ast})$, we can deduce from $\eqref{Riemann2}_1$ a contradiction. We can similarly 
prove \eqref{g2check}.

We notice that $(\tilde{z}(x,t_0),\tilde{w}(x,t_0))$ is contained in 
\begin{align*}
\tilde{S}_{inv}=\left\{(\tilde{z},\tilde{w})\in{\bf R}^2;{\rho}\geq0,\;\tilde{z}\geq-\frac{\left(\bar{\rho}\right)^{\theta}}{\theta}-\varepsilon-E_0,\;\tilde{w}\leq \frac{\left(\bar{\rho}\right)^{\theta}}{\theta}+\varepsilon\right\}.
\end{align*}
Then, we can similarly prove that $\tilde{S}_{inv}$ is an invariant region for the Cauchy problem of \eqref{Riemann1}. Therefore, we conclude \eqref{maintheorem2}.

Although the above argument is formal, it is essential. In fact, we shall implicitly use this argument in the proof of Theorem \ref{goal}. We must next justify the above argument. To do this, we 
introduce a modified Godunov scheme in Section 2. Recently, the various difference schemes are developed in \cite{T1}--\cite{T8}, 
which consist of known functions. On the other hand, the present approximate solutions 
include unknown functions in the form of \eqref{transformation} with constants $\tilde{z},\tilde{w}$ (see \eqref{appro1}).

\section{Construction of Approximate Solutions}
\label{sec:construction-approximate-solutions}
In this section, we construct approximate solutions. Let $T$ be 
any fixed positive constant. In the strip 
$0\leq{t}\leq{T}$, we denote the 
approximate solutions by $u^{\varDelta}(x,t)
=(\rho^{\varDelta}(x,t),m^{\varDelta}(x,t))$. 
We denote the space mesh
lengths by ${\varDelta}x$. 
Using $E_0$ in \eqref{E_0} and $M_0$ in \eqref{M_0}, we take time mesh length ${\varDelta}{t}$ such that 
\begin{align}
	\frac{{\varDelta}x}{{\varDelta}{t}}
	=2(M_0+E_0).
	\label{CFL}
\end{align}
In addition, 
we set 
\begin{align*}
	(j,n)\in2{\bf Z}\times{\bf Z}_{\geq0},
\end{align*}
where ${\bf Z}_{\geq0}=\{0,1,2,3,\ldots\}$.  For simplicity, we use the following terminology
\begin{align}
	\begin{alignedat}{2}
		&x_j=j{\varDelta}x,\;t_n=n{\varDelta}t,\;t_{n.5}=\left(n+\frac12\right){\varDelta}t,
		\;t_{n-}=n{\varDelta}t-0,\;t_{n+}=n{\varDelta}t+0.
	\end{alignedat}
	\label{terminology} 
\end{align}

First we set $u^{\varDelta}(x,t_{0-})=u_0(x)$.

Then, for $j\in 2{\bf Z}$, we define $E_j^0(u)$ by
\begin{align*}
	E_j^0(u)=\frac1{2{\varDelta}x}\int^{x_{j+1}}_{x_{j-1}}
	u^{\varDelta}(x,t_{0-})dx.
\end{align*}

Next, we assume that $u^{\varDelta}(x,t)$ is defined for $t<{t}_{n}$.

Then, for $j\in 2{\bf Z}$, we define $E^n_j(u)$ by 
\begin{align*}
	E^n_j(u)=\frac1{2{\varDelta}x}\int_{x_{j-1}}^{x_{j+1}}u^{\varDelta}(x,t_{n-})dx.
\end{align*}

To determine $u_j^n=(\rho_j^n,m_j^n)$ for $j\in 2{\bf Z}$, we define symbols $I^n_j$ and $L_n$. Let the approximation of $\int^x_{-\infty}J(u(y,t))dy$ be
\begin{align*}
	I^n_j:=
	\int^{x_{j}}_{-\infty}J(E^n(x;u))dx,
\end{align*}where 
\begin{align}
E^n(x;u)=E^n_j(u)\quad x\in[x_{j-1},x_{j+1})
\label{E^n(x;u)}
\end{align}
and $J$ is defined in \eqref{eqn:J}.

Let ${\mathcal D}=(x(t),t)$ denote a discontinuity in 
${u}^{\varDelta}(x,t),\;[\eta_{\ast}]$ and $[q_{\ast}]$ 
denote the jump of $\eta_{\ast}({u}^{ \varDelta}(x,t))$ and $q_{\ast}({u}^{ \varDelta}(x,t))$ across ${\mathcal D}$ from 
left to right, respectively,
\begin{align*}
	&[\eta_{\ast}]=\eta_{\ast}({u}^{\varDelta}(x(t)+0,t))-\eta_{\ast}({u}^{ \varDelta}(x(t)-0,t)),
	\\&
	{[q_{\ast}]=q_{\ast}({u}^{ \varDelta}(x(t)+0,t))-q_{\ast}({u}^{ \varDelta}(x(t)-0,t))},
\end{align*}
where $q_{\ast}(u)$ is the flux of $\eta_{\ast}(u)$ defined by
\begin{align*}
	q_{\ast}(u)=m\left(\frac12\frac{m^2}{\rho^2}+\frac{\rho^{\gamma-1}}{\gamma-1}\right). \end{align*}

Next, to measure the error in the entropy condition and the gap of the 
energy at $t_{n\pm}$, we introduce a functional. To do this, 
we deduce from the Taylor expansion that 
\begin{align}
	\begin{alignedat}{2}
		\eta_{\ast}\left({u}^{\varDelta}(x,t_{n-})\right)
		-
		\eta_{\ast}\left(E^n_j(u)\right)=&
		\nabla\eta_{\ast}(E^n_j(u))\left({u}^{\varDelta}(x,t_{n-})-E^n_j(u)\right)\\&
		+\int^1_0(1-\tau)\cdot{}^t\left({u}^{\varDelta}(x,t_{n-})-E^n_j(u)\right)
		\\&\times\nabla^2\eta_{\ast}\left(E^n_j(u)+\tau\left\{{u}^{\varDelta}(x,t_{n-})-E^n_j(u)\right\}\right)d\tau\\&\times\left({u}^{\varDelta}(x,t_{n-})-E^n_j(u)\right)\\
		=&\nabla\eta_{\ast}(E^n_j(u))\left({u}^{\varDelta}(x,t_{n-})-E^n_j(u)\right)+R^{n}_j(x),
	\end{alignedat}
	\label{Taylor}	
\end{align}where 
\begin{align*}
	R^{n}_j(x)=&\int^1_0(1-\tau)\cdot{}^t\left({u}^{\varDelta}(x,t_{n-})-E^n_j(u)\right)
	\nabla^2\eta_{\ast}\left(E^n_j(u)+\tau\left\{{u}^{\varDelta}(x,t_{n-})-E^n_j(u)\right\}\right)\\&\times\left({u}^{\varDelta}(x,t_{n-})-E^n_j(u)\right)d\tau.
\end{align*}

Then, we define a functional $L_n$ as 
\begin{align}
	\begin{alignedat}{2}
		L_n=&\int^{t_{n}}_{0}\sum_{x\in{\bf R}}\sigma[\eta_{\ast}]-[q_{\ast}]dt
		+\sum^n_{k=0}\int^{\infty}_{-\infty}\left\{\eta_{\ast}({u}^{ \varDelta}(x,t_{k-0}))-\eta_{\ast}(E^k(x;u))\right\}dx\\&+\sum^n_{k=0}\sum_{j\in 2{\bf Z}}\frac1{2{\varDelta}x}\int^{x_{j+1}}_{x_{j-1}}\int^{x}_{x_{j-1}}
	R^k_j(y)dydx,
	\end{alignedat}
	\label{functional discontinuity}	
\end{align}
where the summention in $\sum_{x\in{\bf R}}$
is taken over all discontinuities in ${u}^{ \varDelta}(x,t)$ at a fixed time $t$ over 
$x\in{\bf R}$,
$\sigma$ is the propagating speed of the discontinuities.

Moreover, we set
\begin{align}
M_1=M_0-\delta{\varDelta}t,\quad 
M_{n+1}=
\begin{cases}
M_n-\delta{\varDelta}t,&\text{when }M_{n}+L_n\geq\dfrac{\left(\bar{\rho}\right)^{\theta}}{\theta}+\varepsilon,\\
M_n,&\text{when }M_{n}+L_n<\dfrac{\left(\bar{\rho}\right)^{\theta}}{\theta}+\varepsilon,
\end{cases}
\label{M_n}
\end{align}
where $\delta$ is defined in \eqref{delta}. We notice that $M_n\geq
\dfrac{\left(\bar{\rho}\right)^{\theta}}{\theta}+\varepsilon-\delta{\varDelta}t$.

Using $I^n_j$, $L_n$ and $M_n$, we define $u_j^n$ as follows.

We choose $\mu$ such that $1<\mu<1/(2\theta)$. If 
\begin{align}
	E^n_j(\rho):=
	\frac1{2{\varDelta}x}\int_{x_{j-1}}^{x_{j+1}}\rho^{\varDelta}(x,t_{n-})dx<({\varDelta}x)^{\mu},
\label{eqn:mu}
\end{align} 
we define $u_j^n$ by $u_j^n=(0,0)$;
otherwise, setting\begin{align}
\begin{split}
&{z}_j^n:=
	\max\left\{z(E_j^n(u)),\;-M_n-E_0-L_n+I^n_j\right\},\\
	&w_j^n:=\min\left\{w(E_j^n(u)),\;M_n+L_n+I^n_j\right\}
	,
\end{split}
	\label{def-u^n_j}
\end{align}

we define $u_j^n$ by
\begin{align*}
	u_j^n:=(\rho_j^n,m_j^n):=(\rho_j^n,\rho_j^nv^n_j)
	:=\left(\left\{\frac{\theta(w_j^n-z_j^n)}{2}\right\}
	^{1/\theta},
	\left\{\frac{\theta(w_j^n-z_j^n)}{2}\right\}^{1/\theta}
	\frac{w_j^n+z_j^n}{2}\right).
\end{align*}

\begin{remark}\label{rem:E}
	We find 
	\begin{align}
		\begin{split}
			-M_n-E_0-L_n+I^n_j\leq z(u_j^n),\quad
			{w}(u_j^n)\leq M_n+L_n+I^n_j.
		\end{split}
		\label{remark2.2}
	\end{align}

	This implies that we cut off the parts where 
	$z(E_j^n(u))<-M_n-E_0-L_n+I^n_j$
	and $w(E_j^n(u))>M_n+L_n+I^n_j$
	in  defining $z(u_j^n)$ and 
	${w}(u_j^n)$. Observing \eqref{average}, the order of these cut parts is $o({\varDelta}x)$. The order is so small that we can deduce the compactness and convergence of our approximate solutions.
	
\end{remark}

\subsection{Construction of Approximate Solutions in the Cell}
\label{subsec:construction-approximate-solutions}
We then assume that 
approximate solutions $u^{\varDelta}(x,t)$ are defined in domains $D_1: 
t<{t}_n\quad(
n\in {\bf N})$ and $D_2:
x<x_{j-1}\quad(j\in 2{\bf Z}),\;{t}_n\leq{t}<t_{n+1}$. 
By using $u_j^n$ defined in $D_1$ and $u^{\varDelta}(x,t)$ defined in $
D_2$, we 
construct the approximate solutions  in the cell ${t}_n\leq{t}<{t}_{n+1}\quad 
(n\in{\bf N}),\quad x_{j-1}\leq{x}<x_{j+1}\quad
(j\in 2{\bf Z})$.

We first solve a Riemann problem with initial data $(u_{j-1}^n,u_{j+1}^n)$. 
Call constants $u_{\rm L}(=u_{j-1}^n), u_{\rm M}, u_{\rm R}(=u_{j+1}^n)$ the left, middle and 
right states, respectively. Then the following four cases occur.
\begin{itemize}
	\item {\bf Case 1} A 1-rarefaction wave and a 2-shock arise. 
	\item {\bf Case 2} A 1-shock and a 2-rarefaction wave arise. 
	\item {\bf Case 3} A 1-rarefaction wave and a 2-rarefaction arise.
	\item {\bf Case 4} A 1-shock and a 2-shock arise.
\end{itemize}
We then construct approximate solutions $u^{\varDelta}(x,t)$ by perturbing 
the above Riemann solutions.

Let $\alpha$ be a constant satisfying $1/2<\alpha<1$. We choose 
a positive value $\beta$ small enough.

In this step, we 
consider Case 1 in particular. The constructions of Cases 2--4 are similar 
to that of Case 1. We consider only the case in which $u_{\rm M}$ is away from the vacuum. The other case (i.e., the case where $u_{\rm M}$ is near the vacuum) is a little technical. Therefore, we postpone this case to Appendix \ref{app:vacuum}.

Consider the case where a 1-rarefaction wave and a 2-shock arise as a Riemann 
solution with initial data $(u_j^n,u_{j+1}^n)$. Assume that 
$u_{\rm L},u_{\rm M}$ 
and $u_{\rm M},u_{\rm R}$ are connected by a 1-rarefaction and a 2-shock 
curve, respectively. \vspace*{10pt}\\
{\it Step 1}.\\
In order to approximate a 1-rarefaction wave by a piecewise 
constant {\it rarefaction fan}, we introduce the integer  
\begin{align*}
	p:=\max\left\{[\hspace{-1.2pt}[(z_{\rm M}-z_{\rm L})/({\varDelta}x)^{\alpha}]
	\hspace{-1pt}]+1,2\right\},
\end{align*}
where $z_{\rm L}=z(u_{\rm L}),z_{\rm M}=z(u_{\rm M})$ and $[\hspace{-1.2pt}[x]\hspace{-1pt}]$ is the greatest integer 
not greater than $x$. Notice that
\begin{align}
	p=O(({\varDelta}x)^{-\alpha}).
	\label{order-p}
\end{align}
Define \begin{align*}
	z_1^*:=z_{\rm L},\;z_p^*:=z_{\rm M},\;w_i^*:=w_{\rm L}\;(i=1,\ldots,p),
\end{align*}
and 
\begin{align*}
	z_i^*:=z_{\rm L}+(i-1)({\varDelta}x)^{\alpha}\;(i=1,\ldots,p-1).
\end{align*}
We next introduce the rays $x=(j+1/2){\varDelta}x+\lambda_1(z_i^*,z_{i+1}^*,w_{\rm L})
(t-n{\varDelta}{t})$ separating finite constant states 
$(z_i^*,w_i^*)\;(i=1,\ldots,p)$, 
where  
\begin{align*}
	\lambda_1(z_i^*,z_{i+1}^*,w_{\rm L}):=v(z_i^*,w_{\rm L})
	-S(\rho(z_{i+1}^*,w_{\rm L}),\rho(z_i^*,w_{\rm L})),
\end{align*}
\begin{align*}
	\rho_i^*:=\rho(z_i^*,w_{\rm L}):=\left(\frac{\theta(w_{\rm L}-z_i^*)}2\right)^{1/\theta}\;,
	\quad{v}_i^*:={v}(z_i^*,w_{\rm L}):=\frac{w_{\rm L}+z_i^*}2
\end{align*}
and

\begin{align}
	S(\rho,\rho_0):=\left\{\begin{array}{lll}
		\sqrt{\displaystyle{\frac{\rho(p(\rho)-p(\rho_0))}{\rho_0(\rho-\rho_0)}}}
		,\quad\mbox{if}\;\rho\ne\rho_0,\\
		\sqrt{p'(\rho_0)},\quad\mbox{if}\;\rho=\rho_0.
	\end{array}\right.
	\label{s(,)}
\end{align}

We call this approximated 1-rarefaction wave a {\it 1-rarefaction fan}.

\vspace*{10pt}
{\it Step 2}.\\
In this step, we replace the above constant states  with functions of $x$ and $t$ as follows:

In view of \eqref{transformation}, we construct ${u}^{\varDelta}_1(x,t)$.

We first determine the approximation of $\tilde{z},\tilde{w}$ in \eqref{transformation} 
as follows.
\begin{align*}
	\begin{alignedat}{2}
		\tilde{z}^{\varDelta}_1
		=&z_{\rm L}-
		\int^{x_{j-1}}_{-\infty}
		J(u^{\varDelta}_{n,0}(x))dx,\;
		\tilde{w}^{\varDelta}_1=w_{\rm L}-
		\int^{x_{j-1}}_{-\infty}
		J(u^{\varDelta}_{n,0}(x))dx,
	\end{alignedat}
\end{align*}
where $u^{\varDelta}_{n,0}(x)$
is a piecewise constant function defined by
\begin{align}
	u^{\varDelta}_{n,0}(x)=
		u^n_j,\quad &x\in [x_{j-1},x_{j+1})\quad (j\in 2{\bf Z}).	
	\label{def-u0}
\end{align} 
We set       
\begin{align}
	\begin{alignedat}{2}
		&\check{z}^{\varDelta}_1(x,t)=&&\tilde{z}^{\varDelta}_1
		+\int^{x_{j-1}}_{-\infty}
		J(u^{\varDelta}_{n,0}(x))dx
		+\int^x_{x^{\varDelta}_1}J(u_{\rm L})dy
		\\&&&	
		+\left\{g_1(x,t;u_{\rm L})+V(u_{\rm L})\right\}(t-t_n)
		,\\
		&\check{w}^{\varDelta}_1(x,t)=&&\tilde{w}^{\varDelta}_1
		+ \int^{x_{j-1}}_{-\infty}
		J(u^{\varDelta}_{n,0}(x))dx+\int^x_{x^{\varDelta}_1}J(u_{\rm L})dy
		\\&&&
		+\left\{g_2(x,t;u_{\rm L})+V(u_{\rm L})\right\}
		(t-t_n),
	\end{alignedat}\label{appro1-2}
\end{align}
where $g_1$ and $g_2$ are defined in 
\eqref{inhomo2}, $x^{\varDelta}_1=x_{j-1}$ and 
\begin{align}
	V(u)=q_{\ast}(u)-\frac{\left(\bar{\rho}\right)^{\gamma-1}}{\gamma-1} m.
	\label{V}
\end{align}

From \eqref{appro1-2}, we determine $\check{u}^{\varDelta}_1(x,t)$ by the relation \eqref{relation-Riemann}, that is, 
\begin{align*}
&\check{u}^{\varDelta}_1(x,t)=(\check{\rho}^{\varDelta}_1(x,t),\check{m}^{\varDelta}_1(x,t))=(\check{\rho}^{\varDelta}_1(x,t),\check{\rho}^{\varDelta}_1(x,t)\check{v}^{\varDelta}_1(x,t)),\end{align*} where \begin{align*}
\check{\rho}^{\varDelta}_1(x,t)=\left\{\dfrac{\theta\left(\check{w}^{\varDelta}_1(x,t)-\check{z}^{\varDelta}_1(x,t)\right)}{2}\right\}^{\frac{1}{\theta}},\quad\check{v}^{\varDelta}_1(x,t)=\dfrac{\check{w}^{\varDelta}_1(x,t)+\check{z}^{\varDelta}_1(x,t)}{2}.
\end{align*} 

Using $\check{u}^{\varDelta}_1(x,t)$, we next define ${u}^{\varDelta}_1(x,t)$ as follows. 
\begin{align}
	\begin{alignedat}{2}
		&{z}^{\varDelta}_1(x,t)=&&\tilde{z}^{\varDelta}_1
		+\int^{x_{j-1}}_{-\infty}
		J(u^{\varDelta}_{n,0}(x))dx+\int^x_{x^{\varDelta}_1}
		J(\check{u}^{\varDelta}_1(y,t))dy\\&&&	+
		\left\{g_1(x,t;\check{u}^{\varDelta}_1)+V(u_{\rm L})\right\}(t-t_n)
		,\\
		&{w}^{\varDelta}_1(x,t)=&&\tilde{w}^{\varDelta}_1
		+\int^{x_{j-1}}_{-\infty}
		J(u^{\varDelta}_{n,0}(x))dx+\int^x_{x^{\varDelta}_1}
		J(\check{u}^{\varDelta}_1(y,t))dy\\&&&
		+\left\{g_2(x,t;\check{u}^{\varDelta}_1)+V(u_{\rm L})\right\}(t-t_n).
	\end{alignedat}\label{appro1}
\end{align}
From \eqref{appro1}, we determine ${u}^{\varDelta}_1(x,t)$ by the relation \eqref{relation-Riemann}. 
\begin{remark}${}$
	\begin{enumerate}
		\item We notice that approximate solutions ${z}^{\varDelta}_1,{w}^{\varDelta}_1$ 
		and $\tilde{z}^{\varDelta}_1,\tilde{w}^{\varDelta}_1$ correspond to 
		$z,w$ and $\tilde{z},\tilde{w}$ in \eqref{transformation}, respectively.
		\item 
		For $t_n<t<t_{n+1}$, our approximate solutions will satisfy 
		\begin{align}
			\begin{alignedat}{2} \int^{x_{j-1}}_{-\infty}&
				J(u^{\varDelta}(x,t_{n+1-}))dx+\int^{t_{n+1}}_{t_n}\sum_{x\leq x_{j-1}}(\sigma[\eta_{\ast}]-[q_{\ast}])dt\\
				&=\int^{x_{j-1}}_{-\infty}
				J(u^{\varDelta}_{n,0}(x))dx+V(u_{\rm L}){\varDelta}t
				+o({\varDelta}x).
			\end{alignedat}
			\label{mass-conservation}	
		\end{align}
		In \eqref{appro1}, we thus employ the right hand side of \eqref{mass-conservation} 
		instead of the left hand side.
		\item Our construction of approximate solutions uses the iteration method twice (see \eqref{appro1-2} and \eqref{appro1}) 
		to deduce \eqref{iteration}. 
	\end{enumerate}

\end{remark}

First, by the implicit function theorem, we determine a propagation speed $\sigma_2$ and $u_2=(\rho_2,m_2)$ such that 
\begin{itemize}
	\item[(1.a)] $z_2:=z(u_2)=z^*_2$
	\item[(1.b)] the speed $\sigma_2$, the left state ${u}^{\varDelta}_1(x_2,t_{n.5})$ and the right state $u_2$ satisfy the Rankine--Hugoniot conditions, i.e.,
	\begin{align*}
		f(u_2)-f({u}^{\varDelta}_1(x^{\varDelta}_2(t_{n.5}),t_{n.5}))=\sigma_2(u_2-{u}^{\varDelta}_1(x^{\varDelta}_2(t_{n.5}),t_{n.5})),
	\end{align*}
\end{itemize}
where $x^{\varDelta}_2(t)
=x_j+
\sigma_2(t-t_n)$. Then we fill up by ${u}^{\varDelta}_1(x)$ the sector where $t_n\leq{t}<t_{n+1},x_{j-1}\leq{x}<x^{\varDelta}_2(t)$ (see Figure \ref{case1-1cell}).

\begin{figure}[htbp]
	\begin{center}
		\hspace{-2ex}
		\includegraphics[scale=0.3]{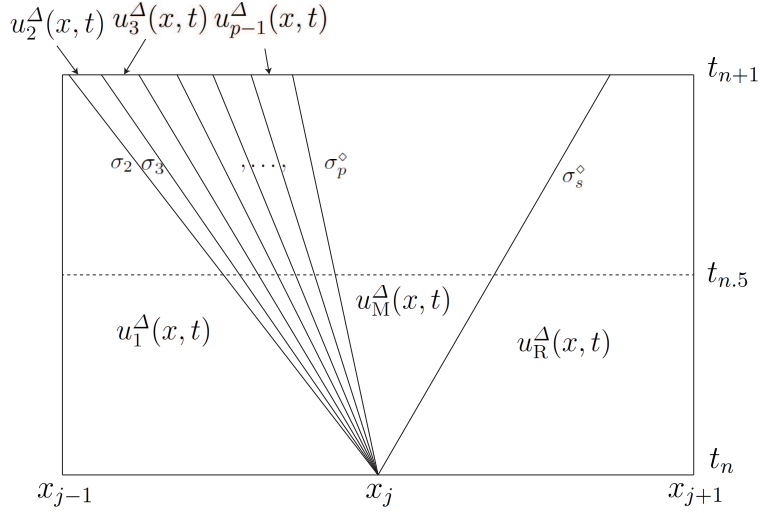}
	\end{center}
	\caption{The approximate solution in the case where a 1-rarefaction and 
		a 2-shock arise in the cell.}
	\label{case1-1cell}
\end{figure}

Assume that $u_k$, ${u}^{\varDelta}_k(x,t)$, a propagation speed $\sigma_k$ and $x^{\varDelta}_{k}(t)$  are defined. Then we similarly determine
$\sigma_{k+1}$ and $u_{k+1}=(\rho_{k+1},m_{k+1})$ such that 
\begin{itemize}
	\item[($k$.a)] $z_{k+1}:=z(u_{k+1})=z^*_{k+1}$,
	\item[($k$.b)] $\sigma_{k}<\sigma_{k+1}$,
	\item[($k$.c)] the speed 
	$\sigma_{k+1}$, 
	the left state ${u}^{\varDelta}_k(x^{\varDelta}_{k+1}(t_{n.5}),t_{n.5})$ and the right state $u_{k+1}$ satisfy 
	the Rankine--Hugoniot conditions, 
\end{itemize}
where $x^{\varDelta}_{k+1}(t)=x_j+\sigma_{k+1}(t-t_n)$. Then we fill up by ${u}^{\varDelta}_k(x,t)$ the sector where $t_n\leq{t}<t_{n+1},x^{\varDelta}_{k}(t)\leq{x}<x^{\varDelta}_{k+1}(t)$ (see Figure \ref{case1-1cell}).

We construct ${u}^{\varDelta}_{k+1}(x,t)$ as follows.

We first determine
\begin{align*}&
	\begin{alignedat}{2}
		&\tilde{z}^{\varDelta}_{k+1}=&&z_{k+1}-\int^{x_{j-1}}_{-\infty}
		J(u^{\varDelta}_{n,0}(x))dx-V(u_{\rm L})\frac{{\varDelta}t}{2}
		-\sum^{k}_{l=1}\int^{x^{\varDelta}_{l+1}(t_{n.5})}_{x^{\varDelta}_l(t_{n.5})}
		J(u^{\varDelta}_l(x,t_{n.5}))dx
		,
	\end{alignedat}\\
	&\begin{alignedat}{2}
		&\tilde{w}^{\varDelta}_{k+1}=&&w_{k+1}-\int^{x_{j-1}}_{-\infty}
		J(u^{\varDelta}_{n,0}(x))dx-V(u_{\rm L})\frac{{\varDelta}t}{2}-\sum^{k}_{l=1}\int^{x^{\varDelta}_{l+1}(t_{n.5})}_{x^{\varDelta}_l(t_{n.5})}
		J(u^{\varDelta}_l(x,t_{n.5}))dx
		,
	\end{alignedat}
\end{align*}
where $x^{\varDelta}_1(t)=x_{j-1},\;x^{\varDelta}_l(t)=x_j+\sigma_l(t-t_n)\quad
(l=2,3,\ldots,k+1)$ and $t_{n.5}$ is defined in \eqref{terminology}.

We next define $\check{u}^{\varDelta}_{k+1}$ as follows.
\begin{align*}
	\begin{alignedat}{2}
		&\check{z}^{\varDelta}_{k+1}(x,t)=&&\tilde{z}^{\varDelta}_{k+1}+\int^{x_{j-1}}_{-\infty}
		J(u^{\varDelta}_{n,0}(x))dx+V(u_{\rm L})(t-t_n)
		+\sum^{k}_{l=1}
		\int^{x^{\varDelta}_{l+1}(t)}_{x^{\varDelta}_l(t)}
		J(u^{\varDelta}_l(x,t))dx\\&&&+\int^x_{x^{\varDelta}_{k+1}(t)}
		J(u_{k+1})dy
		+g_1(x,t;u_{k+1})(t-t_{n.5})
		+\int^{t}_{t_{n.5}}\hspace*{0ex}\sum_{\substack{x_{j-1}\leq y \leq x}}\hspace*{-1ex}(\sigma[\eta_{\ast}]-[q_{\ast}])ds
		,\end{alignedat}
\end{align*}\begin{align*}
	\begin{alignedat}{2}
		&\check{w}^{\varDelta}_{k+1}(x,t)=&&\tilde{w}^{\varDelta}_{k+1}
		+ \int^{x_{j-1}}_{-\infty}
		J(u^{\varDelta}_{n,0}(x))dx+V(u_{\rm L})(t-t_n)
		+\sum^{k}_{l=1}
		\int^{x^{\varDelta}_{l+1}(t)}_{x^{\varDelta}_l(t)}
		J(u^{\varDelta}_l(x,t))dx\\&&&
		+\int^x_{x^{\varDelta}_{k+1}(t)}
		J(u_{k+1})dy
		+g_2(x,t;u_{k+1})(t-t_{n.5})
		+\int^{t}_{t_{n.5}}\hspace*{0ex}\sum_{\substack{x_{j-1}\leq y \leq x}}\hspace*{-1ex}(\sigma[\eta_{\ast}]-[q_{\ast}])ds.
	\end{alignedat}
\end{align*}
From the above, we determine $\check{u}^{\varDelta}_{k+1}(x,t)$ by the relation \eqref{relation-Riemann}.

Finally, using $\check{u}^{\varDelta}_{k+1}(x,t)$, we define ${u}^{\varDelta}_{k+1}(x,t)$ as follows.

\begin{align}
	\begin{alignedat}{2}
		&{z}^{\varDelta}_{k+1}(x,t)&=&\tilde{z}^{\varDelta}_{k+1}+ \int^{x_{j-1}}_{-\infty}
		J(u^{\varDelta}_{n,0}(x))dx+V(u_{\rm L})(t-t_n)
		+\sum^{k}_{l=1}
		\int^{x^{\varDelta}_{l+1}(t)}_{x^{\varDelta}_l(t)}
		J(u^{\varDelta}_l(x,t))dx\\&&&
		+\int^x_{x^{\varDelta}_{k+1}(t)}
		J(\check{u}^{\varDelta}_{k+1}(y,t))dy
		+g_1(x,t;\check{u}^{\varDelta}_{k+1})(t-t_{n.5})
\\&&&
		+\int^{t}_{t_{n.5}}\hspace*{0ex}\sum_{\substack{x_{j-1}\leq y \leq x}}\hspace*{-1ex}(\sigma[\eta_{\ast}]-[q_{\ast}])ds
		,\\
		&{w}^{\varDelta}_{k+1}(x,t)&=&\tilde{w}^{\varDelta}_{k+1}
		+ \int^{x_{j-1}}_{-\infty}
		J(u^{\varDelta}_{n,0}(x))dx+V(u_{\rm L})(t-t_n)
		+\sum^{k}_{l=1}
		\int^{x^{\varDelta}_{l+1}(t)}_{x^{\varDelta}_l(t)}
		J(u^{\varDelta}_l(x,t))dx\\&&&+\int^x_{x^{\varDelta}_{k+1}(t)}J(
		\check{u}^{\varDelta}_{k+1}(y,t))dy
		+g_2(x,t;\check{u}^{\varDelta}_{k+1})(t-t_{n.5})
\\&&&
		+\int^{t}_{t_{n.5}}\hspace*{0ex}\sum_{\substack{x_{j-1}\leq y \leq x}}\hspace*{-1ex}(\sigma[\eta_{\ast}]-[q_{\ast}])ds.
	\end{alignedat}
	\label{appr-k}
\end{align}
From \eqref{appr-k}, we determine ${u}^{\varDelta}_{k+1}(x,t)$ by the relation \eqref{relation-Riemann}.

By induction, we define $u_i$, ${u}^{\varDelta}_i(x,t)$ and $\sigma_i$ $(i=1,\ldots,p-1)$.
Finally, we determine a propagation speed $\sigma_p$ and $u_p=(\rho_p,m_p)$ such that
\begin{itemize}
	\item[($p$.a)] $z_p:=z(u_p)=z^*_p$,
	\item[($p$.b)] the speed $\sigma_p$, 
	and the left state ${u}^{\varDelta}_{p-1}(x^{\varDelta}_{p}(t_{n.5}),t_{n.5})$ and the right state $u_p$ satisfy the Rankine--Hugoniot conditions, 
\end{itemize}where $x^{\varDelta}_{p}(t)=x_j+\sigma_{p}(t-t_n)$. 
We then fill up by ${u}^{\varDelta}_{p-1}(x,t)$ and $u_p$ the sector where
$t_n\leq{t}<t_{n+1},x^{\varDelta}_{p-1}(t)
\leq{x}<x^{\varDelta}_{p}(t)$ 
and the line $t_n\leq{t}<t_{n+1},x=x^{\varDelta}_{p}(t)$, respectively.

Given $u_{\rm L}$ and $z_{\rm M}$ with $z_{\rm L}\leq{z}_{\rm M}$, we denote 
this piecewise functions of $x$ and $t$ 1-rarefaction wave by 
$R_1^{\varDelta}(u_{\rm L},z_{\rm M},x,t)$.

On the other hand, we construct ${u}^{\varDelta}_{\rm R}(x,t)$ as follows. 

We first set 
\begin{align*}
	\begin{alignedat}{2}
		\tilde{z}^{\varDelta}_{\rm R}
		=z_{\rm R}- \int^{x_{j+1}}_{-\infty}
		J(u^{\varDelta}_{n,0}(x))dx,\;
		\tilde{w}^{\varDelta}_{\rm R}
		=w_{\rm R}- \int^{x_{j+1}}_{-\infty}
		J(u^{\varDelta}_{n,0}(x))dx.
	\end{alignedat}
\end{align*}
We next construct $\check{u}^{\varDelta}_{\rm R}$
\begin{align*}
	\begin{alignedat}{2}
		\check{z}^{\varDelta}_{\rm R}(x,t)&&=&\tilde{z}^{\varDelta}_{\rm R}
		+ \int^{x_{j+1}}_{-\infty}
		J(u^{\varDelta}_{n,0}(x))dx+V(u_{\rm R})(t-t_n)
		\\&&&+\int^x_{x_{j+1}}J(u_{\rm R})dy	+g_1(x,t;u_{\rm R})(t-t_n)
		,\\
		\check{w}^{\varDelta}_{\rm R}(x,t)&&=&\tilde{w}^{\varDelta}_{\rm R}
		+\int^{x_{j+1}}_{-\infty}
		J(u^{\varDelta}_{n,0}(x))dx+V(u_{\rm R})(t-t_n)
		\\&&&+\int^x_{x_{j+1}}J(u_{\rm R})dy+g_2(x,t;u_{\rm R})(t-t_n).
	\end{alignedat}
\end{align*}
From the above, we determine $\check{u}^{\varDelta}_{\rm R}(x,t)$ by the relation \eqref{relation-Riemann}.

Using $\check{u}^{\varDelta}_{\rm R}(x,t)$, we define ${u}^{\varDelta}_{\rm R}(x,t)$ as follows. \begin{align}
	\begin{alignedat}{2}
		{z}^{\varDelta}_{\rm R}(x,t)&=&&\tilde{z}^{\varDelta}_{\rm R}
		+ \int^{x_{j+1}}_{-\infty}
		J(u^{\varDelta}_{n,0}(x))dx+V(u_{\rm R})(t-t_n)\\&&&+\int^x_{x_{j+1}}J(\check{u}_{\rm R}(y,t))dy
		+g_1(x,t;\check{u}_{\rm R})(t-t_n)
		,\\
		{w}^{\varDelta}_{\rm R}(x,t)&=&&\tilde{w}^{\varDelta}_{\rm R}
		+ \int^{x_{j+1}}_{-\infty}
		J(u^{\varDelta}_{n,0}(x))dx+V(u_{\rm R})(t-t_n)\\&&&+\int^x_{x_{j+1}}J(\check{u}_{\rm R}(y,t))dy
		+g_2(x,t;\check{u}_{\rm R})(t-t_n).
	\end{alignedat}
	\label{appr-R}
\end{align}
From \eqref{appr-R}, we determine ${u}^{\varDelta}_{\rm R}(x,t)$ by the relation \eqref{relation-Riemann}. 

Now we fix ${u}^{\varDelta}_{\rm R}(x,t)$ and ${u}^{\varDelta}_{p-1}(x,t)$. 
Let $\sigma_s$ be 
the propagation speed of the 2-shock connecting $u_{\rm M}$ and $u_{\rm R}$.
Choosing ${\sigma}^{\diamond}_p$ near to $\sigma_p$, ${\sigma}^{\diamond}_s$ 
near to 
$\sigma_s$ and $u^{\diamond}_{\rm M}$ near to $u_{\rm M}$, we fill up by ${u}^{\varDelta}_{\rm M}(x,t)$ the gap between $x=x_j+{\sigma}^{\diamond}_{p}
(t-{t}_n)$ and $x=x_j+{\sigma}^{\diamond}_s(t-{t}_n)$, such that 
\begin{itemize}
	\item[(M.a)] $\sigma_{p-1}<\sigma^{\diamond}_p<\sigma^{\diamond}_s$, 
	\item[(M.b)] the speed ${\sigma}^{\diamond}_p$, the left and right states 
	${u}^{\varDelta}_{p-1}(x^{\diamond}_{p},t_{n.5}),{u}^{\varDelta}_{\rm M}(x^{\diamond}_{p},t_{n.5})$ 
	satisfy the Rankine--Hugoniot conditions,
	\item[(M.c)] the speed ${\sigma}^{\diamond}_s$, the left and right 
	states ${u}^{\varDelta}_{\rm M}(x^{\diamond}_{s},t_{n.5}),{u}^{\varDelta}_{\rm R}(x^{\diamond}_{s},t_{n.5})$ satisfy the Rankine--Hugoniot conditions, 
\end{itemize}
where $x^{\diamond}_{p}:=x_j+\sigma^{\diamond}_{p}{\varDelta}
/2$, $x^{\diamond}_s:=x_j+\sigma^{\diamond}_s{\varDelta}
/2$ and ${u}^{\varDelta}_{\rm M}(x,t)$ defined as follows.

We first set
\begin{align*}
	\begin{alignedat}{2}
		\tilde{z}^{\varDelta}_{\rm M}&=&&z^{\diamond}_{\rm M}
		- \int^{x_{j+1}}_{-\infty}
		J(u^{\varDelta}_{n,0}\left(x\right))dx- V(u_{\rm R})
		\frac{{\varDelta}t}{2}-\int^{x^{\varDelta}_{\rm R}(t_{n.5})}_{x_{j+1}}
		J(u^{\varDelta}_{\rm R}(x,t_{n.5}))dx
		,\\	
		\tilde{w}^{\varDelta}_{\rm M}&=&&w^{\diamond}_{\rm M}-\int^{x_{j+1}}_{-\infty}
		J(u^{\varDelta}_{n,0}\left(x\right))dx- V(u_{\rm R})
		\frac{{\varDelta}t}{2}-\int^{x^{\varDelta}_{\rm R}(t_{n.5})}_{x_{j+1}}
		J(u^{\varDelta}_{\rm R}(x,t_{n.5}))dx,
	\end{alignedat}
\end{align*}
where $x^{\varDelta}_{\rm R}(t)=x_j+\sigma^{\diamond}_{s}(t-t_n)$.

We construct $\check{u}^{\varDelta}_{\rm M}$
\begin{align*}
	\begin{alignedat}{2}
		\check{z}^{\varDelta}_{\rm M}(x,t)&=&&\tilde{z}^{\varDelta}_{\rm M}
		+ \int^{x_{j+1}}_{-\infty}
		J(u^{\varDelta}_{n,0}(x))dx+V(u_{\rm R})(t-t_n)+\int^{x^{\varDelta}_{\rm R}(t)}_{x_{j+1}}J(u^{\varDelta}_{\rm R}(x,t))dy
		\\&&&	+\int_{x^{\varDelta}_{\rm R}(t)}^{x}J(u_{\rm M})dy
		+g_1(x,t;u_{\rm M})(t-t_{n.5})
\\&&&
-\int^{t}_{t_{n.5}}\hspace*{0ex}\sum_{\substack{x\leq 
				y\leq x_{j+1}}}\hspace*{-1ex}(\sigma[\eta_{\ast}]-[q_{\ast}])ds
		,\\
		\check{w}^{\varDelta}_{\rm M}(x,t)&=&&\tilde{w}^{\varDelta}_{\rm M}
		+ \int^{x_{j+1}}_{-\infty}
		J(u^{\varDelta}_{n,0}(x))dx+V(u_{\rm R})(t-t_n)+\int^{x^{\varDelta}_{\rm R}(t)}_{x_{j+1}}J(u^{\varDelta}_{\rm R}(x,t))dy\\&&&
		+\int_{x^{\varDelta}_{\rm R}(t)}^{x}J(u_{\rm M})dy	+g_2(x,t;u_{\rm M})(t-t_{n.5})
\\&&&
		-\int^{t}_{t_{n.5}}\hspace*{0ex}\sum_{\substack{x\leq 
				y\leq x_{j+1}}}\hspace*{-1ex}(\sigma[\eta_{\ast}]-[q_{\ast}])ds.
	\end{alignedat}
\end{align*}
From the above, we determine $\check{u}^{\varDelta}_{\rm M}(x,t)$ by the relation \eqref{relation-Riemann}.

Using $\check{u}^{\varDelta}_{\rm M}(x,t)$, we next define ${u}^{\varDelta}_{\rm M}(x,t)$ as follows. \begin{align}
	\begin{alignedat}{2}
		{z}^{\varDelta}_{\rm M}(x,t)&=&&\tilde{z}^{\varDelta}_{\rm M}
		+ \int^{x_{j+1}}_{-\infty}
		J(u^{\varDelta}_{n,0}(x))dx+V(u_{\rm R})(t-t_n)+\int^{x^{\varDelta}_{\rm R}(t)}_{x_{j+1}}J(u^{\varDelta}_{\rm R}(x,t))dy\\&&&	
		+\int_{x^{\varDelta}_{\rm R}(t)}^{x}J(\check{u}^{\varDelta}_{\rm M}(y,t))dy+g_1(x,t;\check{u}^{\varDelta}_{\rm M})(t-t_{n.5})
\\&&&
		-\int^{t}_{t_{n.5}}\hspace*{0ex}\sum_{\substack{x\leq 
				y\leq x_{j+1}}}\hspace*{-1ex}(\sigma[\eta_{\ast}]-[q_{\ast}])ds
		,
		\\	{w}^{\varDelta}_{\rm M}(x,t)&=&&\tilde{w}^{\varDelta}_{\rm M}
		+\int^{x_{j+1}}_{-\infty}
		J(u^{\varDelta}_{n,0}(x))dx+V(u_{\rm R})(t-t_n)+\int^{x^{\varDelta}_{\rm R}(t)}_{x_{j+1}}J(u^{\varDelta}_{\rm R}(x,t))dy\\&&&	
		+\int_{x^{\varDelta}_{\rm R}(t)}^{x}J(\check{u}^{\varDelta}_{\rm M}(y,t))dy+g_2(x,t;\check{u}^{\varDelta}_{\rm M})(t-t_{n.5})
		\\&&&
		-\int^{t}_{t_{n.5}}\hspace*{0ex}\sum_{\substack{x\leq 
				y\leq x_{j+1}}}\hspace*{-1ex}(\sigma[\eta_{\ast}]-[q_{\ast}])ds.\end{alignedat}
	\label{appr-M}
\end{align}
From \eqref{appr-M}, we determine ${u}^{\varDelta}_{\rm M}(x,t)$ by the relation \eqref{relation-Riemann}.

We denote this approximate Riemann solution, which consists of \eqref{appr-k}, \eqref{appr-R}, \eqref{appr-R}
, by ${u}^{\varDelta}(x,t)$. The validity of the above construction is demonstrated in \cite[Appendix A]{T1}.

\begin{remark}\label{rem:middle-time}
	${u}^{\varDelta}(x,t)$ satisfies the Rankine--Hugoniot conditions
	at the middle time of the cell, $t=t_{n.5}$.
\end{remark}


\begin{remark}\label{rem:approximate}
	The approximate solution $u^{\varDelta}(x,t)$ is piecewise smooth in each of the 
	divided parts of the cell. Then, in the divided part, $u^{\varDelta}(x,t)$ satisfies
	\begin{align*}
		(u^{\varDelta})_t+f(u^{\varDelta})_x-g(x,u^{\varDelta})=o(1).
	\end{align*}
\end{remark}

To deduce that $L_{n}$ is uniformly bounded, we prove the following lemma.
\begin{lemma}\label{lem:error}
	\begin{align}
		0&\leq\sum^n_{k=0}\int^{\infty}_{-\infty}\left\{\eta_{\ast}({u}^{ \varDelta}(x,t_{k-0}))-\eta_{\ast}(E^k(x;u))\right\}dx+\int^{t_n}_{0}\sum_{x\in{\bf R}}(\sigma[\eta_{\ast}]-[q_{\ast}])dt\label{lemm 3.5.1}\\
		&=\sum^n_{k=0}\sum_{j\in 2{\bf Z}}\int^{x_{j+1}}_{x_{j-1}}R^n_j(x)dx+\int^{t_n}_{0}\sum_{x\in{\bf R}}(\sigma[\eta_{\ast}]-[q_{\ast}])dt+o({\varDelta}x)
		\label{lemm 3.5.2}
		\\ 
		&=
		\int^{\infty}_{-\infty}\left\{J\left({u}^{\varDelta}(x,t_{0-})\right)-J\left({u}^{\varDelta}(x,t_{n+})\right)\right\}dx
		+o({\varDelta}x)\label{lemm 3.5.3}
		\\
		&\leq\int^{\infty}_{-\infty}J(u_0(x))dx+o({\varDelta}x).\label{lemm 3.5.4}
	\end{align}
	where $o({\varDelta}x)$ depends only on $M_0,E_0$ and $T$.

\begin{align}
L_n\leq C,\label{lemm 3.5.5}
\end{align}
where $C$ depends only on initial data. 
\end{lemma}
\begin{proof}
	We recall that our approximate solutions are constructed 
	in $[0,T]$ for any fixed positive constant $T$. From \eqref{initial bar rho} and the finite propagation, we find that 
	our approximate solutions are $(\bar{\rho},0)$ outside a finite interval.

	First, from the Jensen inequality and the entropy condition, we 
	obtain \eqref{lemm 3.5.1}. 
	
	Second, from \eqref{Taylor}, we have \eqref{lemm 3.5.2}.

	Finally, we consider \eqref{lemm 3.5.3}. From the similar argument to \cite[(6.10)]{T1}, taking 
	$J(u)$ as $\eta$ in \cite{T1}, we have 
	\begin{align*}
		&\sum^n_{k=0}\int^{\infty}_{-\infty}\left\{\eta_{\ast}({u}^{ \varDelta}(x,t_{k-0}))-\eta_{\ast}(
		{u}^{\varDelta}_{n,0}(x))\right\}dx+\int^{t_n}_{0}\sum_{x\in{\bf R}}(\sigma[\eta_{\ast}]-[q_{\ast}])dt\\
		&=
		\int^{\infty}_{-\infty}\left\{J\left({u}^{\varDelta}(x,t_{0-})\right)-J\left({u}^{\varDelta}(x,t_{n+})\right)\right\}dx
		+o({\varDelta}x).
	\end{align*}
	On the other hand, \eqref{remark2.2} and Theorem \ref{thm:average}, we 
	find that $u^n_j=E^n_j(u)+o({\varDelta}x)$. Recalling \eqref{E^n(x;u)} and \eqref{def-u0}, we have 
	\begin{align*}
		&\sum^n_{k=0}\int^{\infty}_{-\infty}\left\{\eta_{\ast}({u}^{ \varDelta}(x,t_{k-0}))-\eta_{\ast}(
		{u}^{\varDelta}_{n,0}(x))\right\}dx\\
		&=\sum^n_{k=0}\int^{\infty}_{-\infty}\left\{\eta_{\ast}({u}^{ \varDelta}(x,t_{k-0}))-\eta_{\ast}(
		E^k(x;u))\right\}dx
		+o({\varDelta}x).
	\end{align*}
Finally, observing $R^n_j(x)\geq0$, from \eqref{functional discontinuity} and \eqref{lemm 3.5.4}, 
we have \eqref{lemm 3.5.5}.
\end{proof}

\section{The $L^{\infty}$ estimate of the approximate solutions}
\label{sec:bound}
The aim in this section is to deduce from \eqref{remark2.2} the following
theorem:

\begin{theorem}\label{thm:bound}
For $j\in2{\bf Z}_{\geq0},n\in{\bf Z}_{\geq0}$ and $x_{j-1}\leq x< x_{j+1}$, 
	\begin{align}
		\begin{alignedat}{2}
			&\displaystyle {z}^{\varDelta}(x,t_{n+1-})&\geq&-M_{n+1}-E_0-L_n
			+\int^x_{-\infty}J({u}^{\varDelta}(y,t_{n+1-}))dy-{\it o}({\varDelta}x),\\
			&\displaystyle {w}^{\varDelta}(x,t_{n+1-})
			&\leq&M_{n+1} +L_n+\int^x_{-\infty}J({u}^{\varDelta}(y,t_{n+1-}))dy\\&&&+\int^{t_{n+1}}_{t_n}\sum_{x<x_{j+1}}(\sigma[\eta_{\ast}]-[q_{\ast}])dt+{\it o}({\varDelta}x),
		\end{alignedat}
		\label{goal}
	\end{align}
	where 
	$t_{n-}=n{\varDelta}t-0$, $M_{n+1}$ is defined in \eqref{M_n}, ${\it o}({\varDelta}x)$ 
	depends only on $M_0$ and $E_0$, $\varepsilon$ and $\delta$ are found in 
\eqref{delta}. 
\end{theorem}

\begin{theorem}\label{thm:average}
We assume that $u^{\varDelta}(x,t)$ satisfies \eqref{maintheorem1} and \eqref{goal}.
Then, if $E^{n}_j(\rho)\geq({\varDelta}x)^{\mu}$, it holds that
	\begin{align}
		\begin{split} 
			&-M_{n}-E_0-L_{n}+I^{n}_j
			-{\it o}({\varDelta}x)\leq{z}(E_{j}^{n}(u)),\\
			&w(E^{n}_j(u))\leq M_{n}+L_{n}+I^{n}_j+{\it o}({\varDelta}x),
		\end{split}
		\label{average}
	\end{align}
	where  $j\in2{\bf Z}$ and 
	${\it o}({\varDelta}x)$ depends only on $M_0$. 
\end{theorem}
\eqref{average} is needed to ensures 
\eqref{remark2.2}.

\vspace*{5pt}
Throughout this paper, by the Landau symbols such as $O({\varDelta}x)$,
$O(({\varDelta}x)^2)$ and $o({\varDelta}x)$,
we denote quantities whose moduli satisfy a uniform bound depending only on $M_0$ and $E_0$ unless we specify them.  

Now, in the previous section, we have constructed 
$u^{\varDelta}(x,t)$ in {\bf Case 1}. When we consider $L^{\infty}$ estimates in this case, main difficulty is to obtain $(\ref{goal})_2$ along $R^{\varDelta}_1$. Therefore, we are concerned with $(\ref{goal})_2$ along $R^{\varDelta}_1$.

\subsection{Proof of Theorem \ref{thm:average}.}
We first observe Theorem \ref{thm:average}. 
For $x\in[x_{j-1},x_{j+1}]$, we set 

\begin{align*}
	{z}^{\varDelta}_{\dagger}(x,t_{n-})=&{z}^{\varDelta}(x,t_{n-})-\int^x_{-\infty}
	J\left({u}^{\varDelta}(y,t_{n-})\right)dy+\int^x_{x_{j-1}}
	\eta_{\ast}\left({u}^{\varDelta}(y,t_{n-})\right)dy\\& 
	-\int^x_{x_{j-1}}a^{n}_j{\rho}^{\varDelta}(y,t_{n-})dy
	+\int^x_{x_{j-1}}\dfrac{\left(\bar{\rho}\right)^{\gamma}}{\gamma}dy,\\
	{w}^{\varDelta}_{\dagger}(x,t_{n-})=&{w}^{\varDelta}(x,t_{n-})
	-\int^x_{-\infty}
	J\left({u}^{\varDelta}(y,t_{n-})\right)dy+\int^x_{x_{j-1}}
	\eta_{\ast}\left({u}^{\varDelta}(y,t_{n-})\right)dy \\&
	-\int^x_{x_{j-1}}a^{n}_j{\rho}^{\varDelta}(y,t_{n-})dy
	+\int^x_{x_{j-1}}\dfrac{\left(\bar{\rho}\right)^{\gamma}}{\gamma}dy,
\end{align*}
where 

\begin{align}
a^{n}_j=\dfrac{\partial\eta_{\ast}}{\partial\rho}(E^n_j(u))+
\dfrac{\partial\eta_{\ast}}{\partial m}(E^n_j(u))
\left\{E^n_j(v)-\left(E^n_j(\rho)\right)^{\theta}\right\},\quad E^n_j(v):=\dfrac{E^n_j(m)}{E^n_j(\rho)}.
\label{a^n_j}
\end{align}

Then, by the relation \eqref{relation-Riemann}, we define ${\rho}^{\varDelta}_{\dagger}(x,t_{n-})$ and ${v}^{\varDelta}_{\dagger}(x,t_{n-})$.
We notice that 
\begin{align*}
	{\rho}^{\varDelta}_{\dagger}(x,t_{n-})=&{\rho}^{\varDelta}(x,t_{n-}),\\
	{v}^{\varDelta}_{\dagger}(x,t_{n-})=&{v}^{\varDelta}(x,t_{n-})-\int^x_{-\infty}
	J\left({u}^{\varDelta}(y,t_{n-})\right)dy+\int^x_{x_{j-1}}
	\eta_{\ast}\left({u}^{\varDelta}(y,t_{n-})\right)dy \\&
	-\int^x_{x_{j-1}}a^{n}_j{\rho}^{\varDelta}(y,t_{n-})dy
	+\int^x_{x_{j-1}}\dfrac{\left(\bar{\rho}\right)^{\gamma}}{\gamma}dy.
\end{align*}

Since $L_n$ is positive, $\eqref{average}_2$ is more difficult than $\eqref{average}_1$. 
We thus treat with only $\eqref{average}_2$ in this proof.

\begin{align*}w(E^{n}_j(u))
	=&\frac{\displaystyle \frac1{2{\varDelta}x}\int^{x_{j+1}}_{x_{j-1}}m^{\varDelta}(x,t_{n-})dx+\left(\frac1{2{\varDelta}x}\int^{x_{j+1}}_{x_{j-1}}{\rho}^{\varDelta}(x,t_{n-})dx\right)^{\theta}/\theta
	}{\displaystyle \frac1{2{\varDelta}x}\int^{x_{j+1}}_{x_{j-1}}{\rho}^{\varDelta}(x,t_{n-})dx}
	\\
	=&\frac{\displaystyle\frac1{2{\varDelta}x}\int^{x_{j+1}}_{x_{j-1}}{m}^{\varDelta}_{\dagger}(x,t_{n-})dx+\left(\frac1{2{\varDelta}x}\int^{x_{j+1}}_{x_{j-1}}{\rho}^{\varDelta}_{\dagger}(x,t_{n-})dx\right)^{\theta}/\theta
	}{\displaystyle\frac1{2{\varDelta}x}\int^{x_{j+1}}_{x_{j-1}}{\rho}^{\varDelta}_{\dagger}(x,t_{n-})dx}\\
	&+\frac{\displaystyle \frac1{2{\varDelta}x}\int^{x_{j+1}}_{x_{j-1}}{\rho}^{\varDelta}(x,t_{n-})\left\{	\int^{x_{j-1}}_{-\infty}
		\eta_{\ast}\left({u}^{\varDelta}(y,t_{n-})\right)dy
		\right\}dx
	}{\displaystyle \frac1{2{\varDelta}x}\int^{x_{j+1}}_{x_{j-1}}{\rho}^{\varDelta}(x,t_{n-})dx}
\end{align*}
\begin{align*}
	&-\frac{\displaystyle \frac1{2{\varDelta}x}\int^{x_{j+1}}_{x_{j-1}}{\rho}^{\varDelta}(x,t_{n-})
		\left\{\left(\frac{\left(\bar{\rho}\right)^{\gamma-1}}{\gamma-1}-a^{n}_j\right)	\int^x_{x_{j-1}}
		{\rho}^{\varDelta}(y,t_{n-})dy                \right\}dx
	}{\displaystyle \frac1{2{\varDelta}x}\int^{x_{j+1}}_{x_{j-1}}{\rho}^{\varDelta}(x,t_{n-})dx}
	\\
	&-\frac{\left(\bar{\rho}\right)^{\gamma-1}}{\gamma-1}\int^{x_{j-1}}_{-\infty}{\rho}^{\varDelta}(x,t_{n-})dx
	+\frac{\displaystyle \frac1{2{\varDelta}x}\int^{x_{j+1}}_{x_{j-1}}{\rho}^{\varDelta}(x,t_{n-})\left\{
		\int^{x_{j-1}}_{-\infty}
		\dfrac{\left(\bar{\rho}\right)^{\gamma}}{\gamma}dy\right\}dx
	}{\displaystyle \frac1{2{\varDelta}x}\int^{x_{j+1}}_{x_{j-1}}{\rho}^{\varDelta}(x,t_{n-})dx}
	\\
	=&A_1+A_2-A_3+A_4.	
\end{align*}

We denote the numerator of $A_3$ by $A_{31}$.
From the integration by parts, we have
\begin{align*}
	A_{31=}&\frac1{2{\varDelta}x}\int^{x_{j+1}}_{x_{j-1}}{\rho}^{\varDelta}(x,t_{n-})dx\times\left(\frac{\left(\bar{\rho}\right)^{\gamma-1}}{\gamma-1}\rho-a^{n}_j\right)	\int^{x_{j+1}}_{x_{j-1}}
	{\rho}^{\varDelta}(y,t_{n-})dx\\
	&-\frac1{2{\varDelta}x}\int^{x_{j+1}}_{x_{j-1}}\left\{\int^x_{x_{j-1}}
	{\rho}^{\varDelta}(y,t_{n-})dy \right\}\left(\frac{\left(\bar{\rho}\right)^{\gamma-1}}{\gamma-1}\rho-a^{n}_j\right) {\rho}^{\varDelta}(x,t_{n-})dx\\
	=&\frac1{2{\varDelta}x}\int^{x_{j+1}}_{x_{j-1}}{\rho}^{\varDelta}(x,t_{n-})dx\times\left(\frac{\left(\bar{\rho}\right)^{\gamma-1}}{\gamma-1}\rho-a^{n}_j\right)	\int^{x_{j+1}}_{x_{j-1}}
	{\rho}^{\varDelta}(x,t_{n-})dx-A_{31}.
\end{align*}
We thus obtain 
\begin{align*}
	A_{31}=&\frac12\times
	\frac1{2{\varDelta}x}\int^{x_{j+1}}_{x_{j-1}}{\rho}^{\varDelta}(x,t_{n-})dx\times\left(\frac{\left(\bar{\rho}\right)^{\gamma-1}}{\gamma-1}\rho-a^{n}_j\right)	\int^{x_{j+1}}_{x_{j-1}}
	{\rho}^{\varDelta}(y,t_{n-})dx.
\end{align*}

Therefore, we obtain 
\begin{align}
	\begin{alignedat}{1}
		w(E^{n}_j(u))
		=&\frac{\displaystyle\frac1{2{\varDelta}x}\int^{x_{j+1}}_{x_{j-1}}{m}^{\varDelta}_{\dagger}(x,t_{n-})+\left(\frac1{2{\varDelta}x}\int^{x_{j+1}}_{x_{j-1}}{\rho}^{\varDelta}_{\dagger}(x,t_{n-})dx\right)^{\theta}/\theta
		}{\displaystyle\frac1{2{\varDelta}x}\int^{x_{j+1}}_{x_{j-1}}{\rho}^{\varDelta}_{\dagger}(x,t_{n-})dx}\\
		&+\int^{x_{j-1}}_{-\infty}
	J\left({u}^{\varDelta}(x,t_{n-})\right)dx
		-\dfrac{\frac{\left(\bar{\rho}\right)^{\gamma-1}}{\gamma-1}-a^{n}_j}2	\int^{x_{j+1}}_{x_{j-1}}
		{\rho}^{\varDelta}(x,t_{n-})dx.
	\end{alignedat}\label{lemma 4.1}
\end{align}

Here we introduce the following lemma. The proof is postponed to Appendix A. 
\begin{lemma}\label{lem:average1}
	If
	\begin{align}
		\begin{alignedat}{1}
			\frac1{2{\varDelta}x}\int_{x_{j-1}}
			^{x_{j+1}}{\rho}^{\varDelta}_{\dagger}(x,t_{n-})
			dx\geq ({\varDelta}x)^{\mu}
		\end{alignedat}
		\label{assumption-average2}	
	\end{align}
	and \begin{align}
		{w}^{\varDelta}_{\dagger}(x,t_{n-})\leq& M_{n}+L_{n-1}+\int^x_{x_{j-1}}
		\eta_{\ast}\left({u}^{\varDelta}(y,t_{n-})\right)dy-\int^x_{x_{j-1}}a^{n}_j{\rho}^{\varDelta}(y,t_{n-})dy\nonumber 
		\\&+\int^x_{x_{j-1}}
		\dfrac{\left(\bar{\rho}\right)^{\gamma}}{\gamma}dy+\int^{t_{n}}_{t_{n-1}}\sum_{x<x_{j-1}}(\sigma[\eta_{\ast}]-[q_{\ast}])dt+o({\varDelta}x)\nonumber\\
		=&:
		A(x,t_{n-})+o({\varDelta}x)\hspace*{2.0ex}(x\in [x_{j-1},x_{j+1}]),
		\label{def-A}
	\end{align}
	the following holds
	\begin{align*}
		w(E_j^{n}({u}^{\varDelta}_{\dagger}))
		\leq 
		\bar{A}_j(t_{n-})+o({\varDelta}x),
	\end{align*}
	where $\displaystyle E_j^{n}({u}^{\varDelta}_{\dagger})=\frac1{2{\varDelta}x}\int^{x_{j+1}}_{x_{j-1}}
	u^{\varDelta}_{\dagger}(x,t_{n-})dx,\; \bar{A}_j(t_{n-})=\frac1{2{\varDelta}x}\int^{x_{j+1}}_{x_{j-1}}A(x,t_{n-})dx
	$, the definition of $\mu$ is found in \eqref{eqn:mu}.
\end{lemma}

It follows from \eqref{goal} and this lemma that 
\begin{align}
	\begin{alignedat}{1}
		w(E^{n}_j(u))
		\leq&M_{n}+L_{n-1}+I^{n}_j+\int^{t_{n}}_{t_n}\sum_{y<x_{j-1}}(\sigma[\eta_{\ast}]-[q_{\ast}])dt
		\\
		&+\int^{x_{j-1}}_{-\infty}\left\{\eta_{\ast}\left({u}^{\varDelta}(x,t_{n-})\right)-\eta_{\ast}\left({u}^{\varDelta}_{n,0}(x)\right)\right\}
		dx\\
		&+\frac1{2{\varDelta}x}\int^{x_{j+1}}_{x_{j-1}}
		\int^x_{x_{j-1}}\left\{\eta_{\ast}\left({u}^{\varDelta}(y,t_{n-})\right)
		-\eta_{\ast}\left(E^n_j(u)\right)\right\}dy	dx\\
		&-\frac1{2{\varDelta}x}\int^{x_{j+1}}_{x_{j-1}}\int^x_{x_{j-1}}a^{n}_j\left({\rho}^{\varDelta}(y,t_{n-})-E^n_j(\rho) \right)dydx
		+o({\varDelta}x).
	\end{alignedat}
	\label{lemma 4.1 2}
\end{align}

To complete the proof of Theorem \eqref{average}, we must investigate
\begin{align*}
	\Gamma^{n}_j(y)=&\eta_{\ast}({u}^{\varDelta}(y,t_{n-}))-\eta_{\ast}(E^n_j(u))-a^{n}_j\left({\rho}^{\varDelta}(y,t_{n-})-E^n_j(\rho)\right)
\end{align*}
in \eqref{lemma 4.1 2}, where $a^n_j$ is defined in \eqref{a^n_j}.

We then deduce from \eqref{Taylor} that
\begin{align*}
	\begin{alignedat}{2}
		\Gamma^{n}_j(y)=&
		\dfrac{\partial\eta_{\ast}}{\partial m}(E^n_j(u)){\rho}^{\varDelta}(y,t_{n-})\left(w(y,t_{n-})-w(E^n_j(u))\right)\\&-\dfrac{\partial\eta_{\ast}}{\partial m}(E^n_j(u))\int^1_0(1-\tau)(\theta+1)\left(E^n_j(\rho)+\tau\left\{{\rho}^{\varDelta}(y,t_{n-})-E^n_j(\rho)\right\}\right)^{\theta-1}d\tau\\
		&\times\left({\rho}^{\varDelta}(y,t_{n-})-E^n_j(\rho)\right)^2+R^{n}_j(y)
\end{alignedat}
\end{align*}\begin{align*}
	\begin{alignedat}{2}
		=&
		E^n_j(v){\rho}^{\varDelta}(y,t_{n-})\left(w(y,t_{n-})-w(E^n_j(u)\right)\\&
		-\dfrac{E^n_j(v)}{2}\int^1_0(1-\tau)(\theta+1)\left(E^n_j(\rho)+\tau\left\{{\rho}^{\varDelta}(y,t_{n-})-E^n_j(\rho)\right\}\right)^{\theta-1}d\tau\\
		&\times\left({\rho}^{\varDelta}(y,t_{n-})-E^n_j(\rho)\right)^2+R^{n}_j(y).
	\end{alignedat}
\end{align*}

We thus obtain
\begin{align}
	\begin{alignedat}{2}
		\frac1{2{\varDelta}x}&\int^{x_{j+1}}_{x_{j-1}}\int^x_{x_{j-1}}
		\Gamma^{n}_j(y)dydx\\
		=&\frac1{2{\varDelta}x}\int^{x_{j+1}}_{x_{j-1}}\int^x_{x_{j-1}}E^n_j(v){\rho}^{\varDelta}(y,t_{n-})\left(w(y,t_{n-})-w(E^n_j(u))\right)dydx\\
		&-\frac1{2{\varDelta}x}\int^{x_{j+1}}_{x_{j-1}}\int^x_{x_{j-1}}\dfrac{E^n_j(v)}{2}\int^1_0(1-\tau)(\theta+1)\left(E^n_j(\rho)+\tau\left\{{\rho}^{\varDelta}(y,t_{n-})-E^n_j(\rho)\right\}\right)^{\theta-1}d\tau\\
		&\times\left({\rho}^{\varDelta}(y,t_{n-})-E^n_j(\rho)\right)^2dydx
		+\frac1{2{\varDelta}x}\int^{x_{j+1}}_{x_{j-1}}\int^x_{x_{j-1}}R^{n}_j(y)dydx\\
		=&:B_1+B_2+B_3.
	\end{alignedat}
	\label{lemma 3.4}
\end{align}

If $E^{n}_j(\rho)<({\varDelta}x)^{\mu}$, we find $B_1=o({\varDelta}x)$ and $B_2=o({\varDelta}x)$. 
Therefore, we devote to investigating the case where $E^{n}_j(\rho)\geq({\varDelta}x)^{\mu}$.

If $w(E^n_j(u))\leq M_{n}+L_{n}+I^{n}_j$, \eqref{average} clearly holds.

Otherwise, we consider the following lemma.
\begin{lemma}\label{lem:average2}If 
	\begin{align}
		w(E^n_j(u))> M_{n}+L_{n}+I^{n}_j, 
		\label{Lemma 3.5 1}
	\end{align}the following holds.	
	\begin{align*}
		\frac1{2{\varDelta}x}\int^{x_{j+1}}_{x_{j-1}}\int^x_{x_{j-1}}
		\Gamma^{n}_j(y)dydx\leq&\int^{x_{j+1}}_{x_{j-1}}R^{n}_j(x)dx+o({\varDelta}x).\end{align*}
\end{lemma}
\begin{proof}
	
	From Theorem \ref{thm:bound}, we have $z(E^n_j(u))\geq -M_{n}-L_{n}+I^{n}_j-O({\varDelta}x)$. From 
	\eqref{Lemma 3.5 1}, we find $E^n_j(v)\geq I^{n}_j-O\left({\varDelta}x\right)$ (recall the definition of $E^n_j(v)$ in 
\eqref{a^n_j}
). If $E^n_j(v)<0$, since $I^{n}_j-O\left({\varDelta}x\right)\leq E^n_j(v)\leq 0$, we have 
	\begin{align}
		-E^n_j(v)\leq O\left({\varDelta}x\right).
\label{eqn:E(v)}	
\end{align}

	We first treat with $B_1$ in \eqref{lemma 3.4}. If $E^n_j(v)\geq0$, we deduce from Theorem \ref{thm:bound} and \eqref{Lemma 3.5 1} that
	\begin{align*}
		\begin{alignedat}{2}
			B_1
			\leq&\int^{x_{j+1}}_{x_{j-1}}E^n_j(v){\rho}^{\varDelta}(x,t_{n-})\left(w(x,t_{n-})-w^{n}_j\right)dx
			=o({\varDelta}x).
		\end{alignedat}
	\end{align*}If $E^n_j(v)<0$, from \eqref{eqn:E(v)}, we have 
$B_1=o({\varDelta}x)$.

	We next consider $B_2$. If $E^n_j(v)\geq0$, we find that $B_2\leq0$.
If $E^n_j(v)<0$, from \eqref{eqn:E(v)}, we have 
$B_2=o({\varDelta}x)$. \end{proof}From Lemma \ref{lem:average2}, we can complete the proof of  \eqref{average}. 

\vspace*{0ex}

\subsection{Proof of Theorem \ref{thm:bound}}
We next prove Theorem \ref{thm:bound}.

{\bf Estimates of ${w}^{\varDelta}(x,t)$ along $R^{\varDelta}_1$ in Case 1}
In this step, we estimate ${w}^{\varDelta}(x,t)$ along $R^{\varDelta}_1$ in Case 1
of Section 2. We recall that ${u}^{\varDelta}$ along $R^{\varDelta}_1$ consists of ${u}^{\varDelta}_{k}\quad(k=1,2,3,\ldots,p-1)$.
In this case, ${w}^{\varDelta}(x,t)$ 
has the following properties, which is  proved in \cite[Appendix A]{T1}:
\begin{align}
	{w}^{\varDelta}_{k+1}
	(x^{\varDelta}_{k+1}(t_{n.5}),t_{n.5})=&w_{k+1}
	={w}^{\varDelta}_k(x^{\varDelta}_{k+1}(t_{n.5}),t_{n.5})+{\it O}(({\varDelta}x)^{3\alpha-(\gamma-1)\beta})\nonumber
	\\&\hspace*{22ex}(k=1,\ldots,p-2),
	\label{w_i-w_{i+1}}
\end{align}
where $t_{n.5}$ is defined in \eqref{terminology}.

We first consider $\tilde{w}^{\varDelta}_1$. We recall that
\begin{align*}
	\begin{alignedat}{2}
		\tilde{w}^{\varDelta}_1=w_{\rm L}- \int^{x_{j-1}}_{-\infty}
		J(u^{\varDelta}_{n,0}(x))dx.
	\end{alignedat}
\end{align*}
From \eqref{remark2.2}, we have $\tilde{w}^{\varDelta}_1\leq M_n+L_n$.

Since 
\begin{align}
	\check{u}^{\varDelta}_1(x,t)={u}^{\varDelta}_1(x,t)+
	O(({\varDelta}x)^2),
	\label{iteration}
\end{align}we have 
\begin{align*}
	\begin{alignedat}{2}
		&{w}^{\varDelta}_1(x,t)&=&\tilde{w}^{\varDelta}_1
		+ \int^{x_{j-1}}_{-\infty}
		J(u^{\varDelta}_{n,0}(x))dx+V(u_{\rm L})(t-t_n)+\int^x_{x^{\varDelta}_1}
		J(\check{u}^{\varDelta}_1(y,t))dy\\&&&
		+g_2(x,t;\check{u}^{\varDelta})(t-t_{n})
		\\
		&&\leq&M_n+L_n+ \int^{x_{j-1}}_{-\infty}
		J(u^{\varDelta}_{n,0}(x))dx+V(u_{\rm L})(t-t_n)+\int^x_{x^{\varDelta}_1}
		J({u}^{\varDelta}_1(y,t))dy\\&&&
		+g_2(x,t;{u}^{\varDelta})(t-t_{n})
		+o({\varDelta}x).
	\end{alignedat}
\end{align*}
On the other hand, from the construction of our approximate solutions, 
we observe that ${w}^{\varDelta}_1(x,t)={w}^{\varDelta}_1(x,t_{n-})+O({\varDelta}x)\quad (x_{j-1}\leq x<x_{j+1},\;t_n\leq t<t_{n+1})$.
Separating three cases, we prove $\eqref{goal}_2$. 
\begin{enumerate}
\item If ${w}^{\varDelta}_1(x,t_{n-})<M_n+L_n+I^n_j-\sqrt{{\varDelta}x}$, we obtain  $\eqref{goal}_2$. 
\item If ${w}^{\varDelta}_1(x,t_{n-})\geq M_n+L_n+I^n_j-\sqrt{{\varDelta}x}$ and 
$M_n+L_n\geq \dfrac{\left(\bar{\rho}\right)^{\theta}}{\theta}+\varepsilon$, 
we observe that ${w}^{\varDelta}_1(x,t)
\geq{w}^{\varDelta}_1(x,t_{n-})-O(\sqrt{{\varDelta}x})\geq M_n+L_n-O(\sqrt{{\varDelta}x})\geq\dfrac{\left(\bar{\rho}\right)^{\theta}}{\theta}+\varepsilon-O(\sqrt{{\varDelta}x})\geq\dfrac{\left(\bar{\rho}\right)^{\theta}}{\theta}+\varepsilon/2$, by choosing ${\varDelta}x$ small enough.
From \eqref{delta}, 
we obtain $g_2(x,t;{u}^{\varDelta}_1)<-2\delta$. 
From \eqref{mass-conservation}, we obtain $\eqref{goal}_2$.
\item If ${w}^{\varDelta}_1(x,t_{n-})\geq M_n+L_n+I^n_j-\sqrt{{\varDelta}x}$ and 
$M_n+L_n< \dfrac{\left(\bar{\rho}\right)^{\theta}}{\theta}+\varepsilon$, from \eqref{M_n},
we find that $\dfrac{\left(\bar{\rho}\right)^{\theta}}{\theta}+\varepsilon-\delta {\varDelta}t\leq M_n+L_n$.
Therefore, we have ${w}^{\varDelta}_1(x,t)
\geq{w}^{\varDelta}_1(x,t_{n-})-O(\sqrt{{\varDelta}x})\geq M_n+L_n-O(\sqrt{{\varDelta}x})\geq\dfrac{\left(\bar{\rho}\right)^{\theta}}{\theta}+\varepsilon-O(\sqrt{{\varDelta}x})\geq\dfrac{\left(\bar{\rho}\right)^{\theta}}{\theta}+\varepsilon/2$, by choosing ${\varDelta}x$ small enough.
From \eqref{delta}, 
we obtain $g_2(x,t;{u}^{\varDelta}_1)<-2\delta<0$. 
Recalling \eqref{M_n}, from \eqref{mass-conservation}, we obtain $\eqref{goal}_2$.
\end{enumerate}

Next, we assume that \begin{align}
	\begin{alignedat}{2}
		&{w}^{\varDelta}_k(x,t)
		&\leq&M_n+L_n+\int^{x_{j-1}}_{-\infty}
		J(u^{\varDelta}_{n,0}(x))dx+V(u_{\rm L})(t-t_n)\\&&&+\int^x_{x_{j-1}}
		J({u}^{\varDelta}(y,t))dy
		+\int^{t}_{t_{n.5}}\hspace*{0ex}\sum_{x_{j-1}\leq y< x}\hspace*{-1ex}(\sigma[\eta_{\ast}]-[q_{\ast}])dt\\&&&+(k-1)\cdot{\it O}(({\varDelta}x)^{3\alpha-(\gamma-1)\beta})
	+o({\varDelta}x)
		\label{assumption}
	\end{alignedat}
\end{align}
for $(x,t)\in [x_{j-1},x^{\varDelta}_{k+1}(t))\times[t_n,t_{n+1})$.

We recall that\begin{align*}
	\begin{alignedat}{2}
		&\tilde{w}^{\varDelta}_{k+1}=&&w_{k+1}-\int^{x_{j-1}}_{-\infty}
		J(u^{\varDelta}_{n,0}(x))dx-V(u_{\rm L})\frac{{\varDelta}t}{2}-\sum^{k}_{l=1}\int^{x^{\varDelta}_{l+1}(t_{n.5})}_{x^{\varDelta}_l(t_{n.5})}
		J(u^{\varDelta}_l(x,t_{n.5}))dx
	\end{alignedat}
\end{align*}
and ${u}^{\varDelta}(x,t)$ consists $u^{\varDelta}_l(x,t)$ in 
$x^{\varDelta}_{l}(t)\leq x<x^{\varDelta}_{l+1}(t),\;t_n\leq t<t_{n+1}\quad(l=1,2,3,\ldots,k+1)$.

From \eqref{w_i-w_{i+1}} and \eqref{assumption}, 
we have 
\begin{align*}
	\tilde{w}^{\varDelta}_{k+1}\leq& M_n+L_n
+k\cdot{\it O}(({\varDelta}x)^{3\alpha-(\gamma-1)\beta})+o({\varDelta}x).	
\end{align*}

From a similar argument to ${w}^{\varDelta}_1$, we have
\begin{align*}&\begin{alignedat}{2}
		&{w}^{\varDelta}_{k+1}(x,t)&=&\tilde{w}^{\varDelta}_{k+1}
		+ \int^{x_{j-1}}_{-\infty}
		J(u^{\varDelta}_{n,0}(x))dx+V(u_{\rm L})(t-t_n)
		+\sum^{k}_{l=1}
		\int^{x^{\varDelta}_{l+1}(t)}_{x^{\varDelta}_l(t)}
		J(u^{\varDelta}_l(x,t))dx\\&&&+\int^x_{x^{\varDelta}_{k+1}(t)}J(
		\check{u}^{\varDelta}_{k+1}(y,t))dy
		+g_2(x,t;\check{u}^{\varDelta}_{k+1})(t-t_{n.5})
\\&&&
			+\int^{t}_{t_{n.5}}\hspace*{0ex}\sum_{\substack{x_{j-1}\leq y \leq x}}\hspace*{-1ex}(\sigma[\eta_{\ast}]-[q_{\ast}])ds\\
		&&\leq&M_n+L_n+\int^{x_{j-1}}_{-\infty}
		J(u^{\varDelta}_{n,0}(x))dx+V(u_{\rm L})(t-t_n)+\int^x_{x_{j-1}}J(
		{u}^{\varDelta}(y,t))dy\\&&&+g_2(x,t;\check{u}^{\varDelta}_{k+1})(t-t_{n.5})+\int^{t}_{t_{n.5}}\hspace*{0ex}\sum_{\substack{x_{j-1}\leq y \leq x}}\hspace*{-1ex}(\sigma[\eta_{\ast}]-[q_{\ast}])ds\\
&&&
		+k\cdot{\it O}(({\varDelta}x)^{3\alpha-(\gamma-1)\beta})+o({\varDelta}x)\qquad
		(k=1,2,3,\ldots,p-1).
	\end{alignedat}
\end{align*}
From \eqref{order-p} and \eqref{mass-conservation}, since $\left\{3\alpha-(\gamma-1)\beta\right\}p>1$, we conclude $\eqref{goal}_2$.

\section{Proof of Theorem \ref{thm:main}}

Our approximate solutions satisfy the following propositions holds (these proofs are similar to \cite{T1}--\cite{T3}.).
\begin{proposition}\label{pro:compact}
	The measure sequence
	\begin{align*}
		\eta_{\ast}(u^{\varDelta})_t+q(u^{\varDelta})_x
	\end{align*}
	lies in a compact subset of $H_{\rm loc}^{-1}(\Omega)$ for all weak entropy 
	pair $(\eta_{\ast},q)$, where $\Omega\subset[0,1]\times[0,1]$ is any bounded
	and open set. 
\end{proposition}
\begin{proposition} 
	Assume that the approximate solutions $u^{\varDelta}$ are bounded and satisfy Proposition \ref{pro:compact}. Then there is a convergent subsequence $u^{\varDelta_n}(x,t)$
	in the approximate solutions $u^{\varDelta}(x,t)$ such that
	\begin{equation*}      
		u^{\varDelta_n}(x,t)\rightarrow u(x,t)
		\hspace{2ex}
		\text{\rm a.e.,\quad as\;\;}n\rightarrow \infty.
	\end{equation*} 
	The function $u(x,t)$ is a global entropy solution
	of the Cauchy problem \eqref{IP}. 
\end{proposition} 
Moreover, from Theorem \ref{thm:bound}, the above solution satisfies \eqref{maintheorem2}. Therefore, we can 
prove Theorem \ref{thm:main}.

\appendix
\section{Proof of \eqref{estimate3} and \eqref{estimate4}}\label{app:formal}
\subsection{Proof of \eqref{estimate3}}
First, when $0\leq\rho\leq\bar{\rho}$, we will prove 
\begin{align*}
	&\dfrac{5\gamma-3}{\gamma(\gamma-1)^2}\rho^{\gamma+\theta}
	-\dfrac{2(3\gamma-1)}{\gamma(\gamma-1)^2}\left(\bar{\rho}\right)^{\theta}
	\rho^{\gamma}+
	\dfrac{3-\gamma}{(\gamma-1)^2}\left(\bar{\rho}\right)^{\gamma-1}
	\rho^{\theta+1}
	-\dfrac{3-\gamma}{\gamma(\gamma-1)}\left(\bar{\rho}\right)^{\gamma}
	\rho^{\theta}\vspace*{1ex}\\&+
	\dfrac{2}{\gamma(\gamma-1)}
	\left(\bar{\rho}\right)^{\gamma+\theta}
	\geq0.
\end{align*}
To this, setting $t=\rho/\bar{\rho}$, we consider
\begin{align*}
	f(t)=&(5\gamma-3)t^{3\theta+1}-2(3\gamma-1)t^{2\theta+1}+\gamma(3-\gamma)t^{\theta+1}
	-(3-\gamma)(\gamma-1)t^{\theta}\\
	&+2(\gamma-1),\quad 0\leq t\leq1.
\end{align*}
Separating 3 steps, we will deduce that
$f(t)\geq0,\quad 0\leq t\leq1$.

{\bf Step 1}

First, we consider the neighborhood of $t=0$.
We set $X=t^{\theta}$. Solving two inequalities 
\begin{align*}
	&(5\gamma-3)t^{3\theta+1}-2(3\gamma-1)t^{2\theta+1}+\gamma(3-\gamma)t^{\theta+1}
	\\&
	=t^{\theta+1}\left\{(5\gamma-3)X^2-2(3\gamma-1)X+\gamma(3-\gamma)\right\}
	\geq0
\end{align*}
and 
\begin{align*}
	-(3-\gamma)(\gamma-1)t^{\theta}+2\gamma(\gamma-1)
	=	-(3-\gamma)(\gamma-1)X+2(\gamma-1)\geq0,
\end{align*}
we have $0\leq X\leq \xi$, where $\xi$ is the smaller solution 
of $(5\gamma-3)X^2-2(3\gamma-1)X+\gamma(3-\gamma)=0$. We 
notice that $f(t)\geq0$ in the interval $0\leq X\leq \xi$.

{\bf Step 2}

Next, we consider the neighborhood of $t=1$. 
We find that $f(1)=f'(1)=0$.
On the other hand, from $0\leq t\leq1,\;\gamma>1$, 
we have
\begin{align*}
	4f''(t)=&3(5\gamma-3)(3\gamma-1)(\gamma-1)t^{3\theta-1}-8
	\gamma(\gamma-1)(3\gamma-1)t^{2\theta-1}\\&
	+\gamma(\gamma-1)(\gamma+1)(3-\gamma)t^{\theta-1}+
	(\gamma-1)^2(3-\gamma)^2t^{\theta-2}\\
	\geq&3(5\gamma-3)(3\gamma-1)(\gamma-1)t^{3\theta-1}-8
	\gamma(\gamma-1)(3\gamma-1)t^{2\theta-1}\\&
	+(5\gamma-3)(\gamma-1)(3-\gamma)t^{\theta-1}\\
	\geq&3(5\gamma-3)(3\gamma-1)(\gamma-1)t^{3\theta-1}-8
	\gamma(\gamma-1)(3\gamma-1)t^{2\theta-1}\\&
	+(3\gamma-1)(\gamma-1)(3-\gamma)t^{\theta-1}\\
	=&(3\gamma-1)(\gamma-1)t^{\theta-1}
	\left\{3(5\gamma-3)X^2-8
	\gamma X
	+3-\gamma
	\right\}.
\end{align*}
We thus find that $f(t)\geq0$ in the interval $\eta\leq X\leq1$, 
where $\eta$ is the larger solution of 
$3(5\gamma-3)X^2-8
\gamma X
+3-\gamma=0$.

{\bf Step 3}

Since $\xi<\eta$, from Step 1,2, it suffices to prove $f(t)\geq0$ in 
the interval 
$\xi\leq X\leq \eta$. Observing that
$(5\gamma-3)X^2-2(3\gamma-1)X+\gamma(3-\gamma)\leq0$ in this interval, we have
\begin{align*}
	f(t)=&t^{\theta+1}\left\{(5\gamma-3)X^2-2(3\gamma-1)X+\gamma(3-\gamma)\right\}	-(3-\gamma)(\gamma-1)X+2(\gamma-1)\\
	\geq&X\left\{(5\gamma-3)X^2-2(3\gamma-1)X+\gamma(3-\gamma)\right\}	-(3-\gamma)(\gamma-1)X+2(\gamma-1)\\
	=&(5\gamma-3)X^3-2(3\gamma-1)X^2+(3-\gamma)X+2(\gamma-1)=:g(X).
\end{align*}
Let $\alpha,\;\beta\;(\alpha<\beta)$ be tow solutions of 
$g'(X)=0$. Then, we find that 
$0<\alpha<\xi<\eta<\beta<1$. Moreover, we deduce that 
$g(\eta)>0$. Therefore, we can complete the proof.

\subsection{Proof of \eqref{estimate4}}
Our goal in this appendix is to prove 
\begin{align*}
	\dfrac{\gamma+1}{2\gamma^2(\gamma-1)}\rho^{\gamma+\theta}
	-\dfrac{1}{\gamma-1}\left(\bar{\rho}\right)^{\gamma-1}
	\rho^{\theta+1}+
	\dfrac{\gamma+1}{\gamma^2}\left(\bar{\rho}\right)^{\gamma}
	\rho^{\theta}
	-\dfrac{1}{2\gamma^2}\left(\bar{\rho}\right)^{2\gamma}\dfrac{1}{\rho^{\theta+1}}
	\geq0,
\end{align*}where $\rho\geq\bar{\rho}$. 
To do this, setting $t=\rho/\bar{\rho}$, we prove 
\begin{align*}
	g(t)=&\dfrac{\gamma+1}{2\gamma^2(\gamma-1)}t^{2\gamma}-\dfrac{1}{\gamma-1}t^{\gamma+1}+\dfrac{\gamma+1}{\gamma^2}t^{\gamma}-\dfrac{1}{2\gamma^2}
	\geq0,\quad t\geq1.
\end{align*}
First, we observe that $g(1)=g'(1)=g''(1)=0$. In addition, we find that 
$g'''(t)\geq0,\quad t\geq1$. We thus conclude that $g(t)\geq0$.

\section{Proof of Lemma \ref{lem:average1}}
\begin{proof}
	
	Due to space limitations, we denote $t_{n-}$ by $T$ in this section.

	Set
	\begin{align*}
		\rho^{\varDelta}_{\dagger}(x,T)&:=\hat{\rho}(x,T)\left\{A(x,T)\right\}^{\frac{2}{\gamma-1}},\\
		m^{\varDelta}_{\dagger}(x,T)&:=\hat{m}(x,T)\left\{A(x,T)\right\}^{\frac{\gamma+1}{\gamma-1}},\\
		E_j^{n+1}(\rho^{\varDelta}_{\dagger})&:=\frac1{2{\varDelta}x}\int_{x_{j-1}}
		^{x_{j+1}}\hat{\rho}(x,T)\left\{A(x,T)\right\}^{\frac{2}{\gamma-1}}dx,
		\\
		E_j^{n+1}(m^{\varDelta}_{\dagger})&:=\frac1{2{\varDelta}x}\int_{x_{j-1}}
		^{x_{j+1}}\hat{m}(x,T)\left\{A(x,T)\right\}^{\frac{\gamma+1}{\gamma-1}}dx.
	\end{align*}
	Then, we find that  
	\begin{align}
		{w}(\hat{u}(x,T))\leq 1+{\it o}({\varDelta}x).
		\label{appendix2}
	\end{align}

	Let us  prove 
	\begin{align*}
		w(E_j^{n+1}(\rho^{\varDelta}_{\dagger}),E_j^{n+1}(m^{\varDelta}_{\dagger}))\leq \bar{A}_j(T)+o({\varDelta}x),
	\end{align*}
	where
	\begin{align*}
		\bar{A}_j(T)=\frac1{2{\varDelta}x}\int_{x_{j-1}}
		^{x_{j+1}}A(x,T)dx
	\end{align*}
	and
	\begin{align}
		w&(E_j^{n+1}(\rho^{\varDelta}_{\dagger}),E_j^{n+1}(m^{\varDelta}_{\dagger}))\nonumber\\
		&=E_j^{n+1}(m^{\varDelta}_{\dagger})/E_j^{n+1}(\rho^{\varDelta}_{\dagger})+\{E_j^{n+1}(\rho^{\varDelta}_{\dagger})\}^{\theta}
		/\theta\nonumber\\
		&=\frac{\displaystyle{\frac1{2{\varDelta}x}\int_{x_{j-1}}
				^{x_{j+1}}\hspace{-0.6ex}\hat{m}(x,T)\left\{A(x,T)\right\}^{\frac{\gamma+1}{\gamma-1}}dx}
			+{\left(\frac1{2{\varDelta}x}\int_{x_{j-1}}
				^{x_{j+1}}\hspace{-0.6ex}\hat{\rho}(x,T)
				\left\{A(x,T)\right\}^{\frac{2}{\gamma-1}}
				dx\right)^{\theta+1}}\hspace{-3ex}/{\theta}}
		{\displaystyle{\frac1{2{\varDelta}x}\int_{x_{j-1}}^{x_{j+1}}\hat{\rho}(x,T)	\left\{A(x,T)\right\}^{\frac{2}{\gamma-1}}dx}}.
		\nonumber\\
		\label{lemma3.1-2}
	\end{align}

	{\it Step 1.}\\
	We find
	\begin{align*}
		E_j^{n+1}(\rho^{\varDelta}_{\dagger})&=\frac1{2{\varDelta}x}\int_{x_{j-1}}
		^{x_{j+1}}\hat{\rho}(x,T)\left\{A(x,T)\right\}^{\frac{\gamma+1}{\gamma-1}}\left\{A(x,T)\right\}^{-1}dx\\
		&=\left\{\bar{A}_j(T)\right\}^{-1}
		\frac1{2{\varDelta}x}\int_{x_{j-1}}
		^{x_{j+1}}\hat{\rho}(x,T)\left\{A(x,T)\right\}^{\frac{\gamma+1}{\gamma-1}}dx\\
		&\quad+\frac1{2{\varDelta}x}\int_{x_{j-1}}
		^{x_{j+1}}\hat{\rho}(x,T)\left\{A(x,T)\right\}^{\frac{\gamma+1}{\gamma-1}}
		\times\left(\left\{A(x,T)\right\}^{-1}-\left\{\bar{A}_j(T)\right\}^{-1}    \right)dx
		\\
		&=\left\{\bar{A}_j(T)\right\}^{-1}
		\frac1{2{\varDelta}x}\int_{x_{j-1}}
		^{x_{j+1}}\hat{\rho}(x,T)\left\{A(x,T)\right\}^{\frac{\gamma+1}{\gamma-1}}dx
		\\&\quad-\left\{\bar{A}_j(T)\right\}^{-1}\frac1{2{\varDelta}x}\int_{x_{j-1}}
		^{x_{j+1}}\hat{\rho}(x,T)\left\{A(x,T)\right\}^{\frac{2}{\gamma-1}}
		r(x,T)dx+o({\varDelta}x),
	\end{align*}
	where $\displaystyle r(x,T)=A(x,T)-\bar{A}_j(T)$. Recalling \eqref{def-A}, we notice that 
	$\displaystyle r(x,T)=O({\varDelta}x)$.

	Substituting the above equation for (\ref{lemma3.1-2}), we obtain
	\begin{align}
		\lefteqn{w(E_j^{n+1}(\rho^{\varDelta}_{\dagger}),E_j^{n+1}(m^{\varDelta}_{\dagger}))}\nonumber\\
		=&\frac{\displaystyle{\frac1{2{\varDelta}x}\int_{x_{j-1}}
				^{x_{j+1}}\hspace{-0.6ex}\hat{m}(x,T)\left\{A(x,T)\right\}^{\frac{\gamma+1}{\gamma-1}}dx}
			+{\left(\frac1{2{\varDelta}x}\int_{x_{j-1}}
				^{x_{j+1}}\hspace{-0.6ex}\hat{\rho}(x,T)
				\left\{A(x,T)\right\}^{\frac{2}{\gamma-1}}
				dx\right)^{\theta+1}}\hspace{-3ex}/{\theta}}
		{\displaystyle{\left\{\bar{A}_j(T)\right\}^{-1}
				\frac1{2{\varDelta}x}\int_{x_{j-1}}
				^{x_{j+1}}\hat{\rho}(x,T)\left\{A(x,T)\right\}^{\frac{\gamma+1}{\gamma-1}}dx}
		}\nonumber\\
		&+\frac{\displaystyle{\frac1{2{\varDelta}x}\int_{x_{j-1}}
				^{x_{j+1}}\hspace{-0.6ex}\hat{m}(x,T)\left\{A(x,T)\right\}^{\frac{\gamma+1}{\gamma-1}}dx}
			+{\left(\frac1{2{\varDelta}x}\int_{x_{j-1}}
				^{x_{j+1}}\hspace{-0.6ex}\hat{\rho}(x,T)
				\left\{A(x,T)\right\}^{\frac{2}{\gamma-1}}
				dx\right)^{\theta+1}}\hspace{-3ex}/{\theta}}
		{\displaystyle{\left(
				\frac1{2{\varDelta}x}\int_{x_{j-1}}
				^{x_{j+1}}\hat{\rho}(x,T)\left\{A(x,T)\right\}^{\frac{2}{\gamma-1}}dx \right)^2
			}
		}
		\nonumber\\
		&\times \left\{\bar{A}_j(T)\right\}^{-1}\frac1{2{\varDelta}x}\int_{x_{j-1}}
		^{x_{j+1}}\hat{\rho}(x,T)\left\{A(x,T)\right\}^{\frac{2}{\gamma-1}}
		r(x,T)dx
		+o({\varDelta}x).
		\label{lemma3.1-5}
	\end{align}
	
	Set
	\begin{align}
		\lefteqn{\omega:=\frac{2}{\gamma+1}\frac1{\displaystyle{\left(
					\frac1{2{\varDelta}x}\int_{x_{j-1}}
					^{x_{j+1}}\hspace{-0.6ex}\hat{\rho}(x,T)
					\left\{A(x,T)\right\}^{\frac{2}{\gamma-1}}
					dx
					\right)^{\theta}}}}\nonumber\\
		&\times
		\frac{\displaystyle{\frac1{2{\varDelta}x}\int_{x_{j-1}}
				^{x_{j+1}}\hspace{-0.6ex}\hat{m}(x,T)\left\{A(x,T)\right\}^{\frac{\gamma+1}{\gamma-1}}dx}
			+{\left(\frac1{2{\varDelta}x}\int_{x_{j-1}}
				^{x_{j+1}}\hspace{-0.6ex}\hat{\rho}(x,T)
				\left\{A(x,T)\right\}^{\frac{2}{\gamma-1}}
				dx\right)^{\theta+1}}\hspace{-3ex}/{\theta}}
		{\displaystyle{
				\frac1{2{\varDelta}x}\int_{x_{j-1}}
				^{x_{j+1}}\hat{\rho}(x,T)\left\{A(x,T)\right\}^{\frac{2}{\gamma-1}}dx}
		}
		.
		\label{mu}
	\end{align}
	Then assume that the following holds.
	\begin{align}
		(E_j^{n+1}(\rho^{\varDelta}_{\dagger}))^{\theta+1}
		&\leq
		\frac1{2{\varDelta}x}\int_{x_{j-1}}
		^{x_{j+1}}(\hat{\rho}(x,T))^{\theta+1}
		\left\{A(x,T)\right\}^{\frac{\gamma+1}{\gamma-1}}dx\nonumber\\
		&\quad-\frac{\gamma+1}{2}\omega\left\{\bar{A}_j(T)\right\}^{-1}\left(
		\frac1{2{\varDelta}x}\int_{x_{j-1}}
		^{x_{j+1}}\hat{\rho}(x,T)\left\{A(x,T)\right\}^{\frac{2}{\gamma-1}}dx       \right)^{\theta}\nonumber\\
		&\quad\times\left(	\frac1{2{\varDelta}x}\int_{x_{j-1}}
		^{x_{j+1}}\hat{\rho}(x,T)
		\left\{A(x,T)\right\}^{\frac{2}{\gamma-1}}r(x,T)dx\right.\nonumber\\&\left.
		\quad-\frac1{2{\varDelta}x}\int_{x_{j-1}}
		^{x_{j+1}}\hat{\rho}(x,T)
		\left\{A(x,T)\right\}^{\frac{2}{\gamma-1}}dx\frac1{2{\varDelta}x}\int_{x_{j-1}}
		^{x_{j+1}}r(x,T)dx\right)
		\nonumber\\
		&\quad+o({\varDelta}x)\frac1{2{\varDelta}x}\int_{x_{j-1}}
		^{x_{j+1}}\hat{\rho}(x,T)
		\left\{A(x,T)\right\}^{\frac{2}{\gamma-1}}dx.\label{lemma3.1-13}
	\end{align}
	
	This estimate shall be proved in step 2--4.
	Then, substituting (\ref{lemma3.1-13}) for (\ref{lemma3.1-5}),
	we deduce from \eqref{appendix2} that
	\begin{align*}
		w(E_j^{n+1}(\bar{\rho}),E_j^{n+1}(m^{\varDelta}_{\dagger}))
		\leq&\frac{\displaystyle{\frac1{2{\varDelta}x}\int_{x_{j-1}}
				^{x_{j+1}}\hspace{-1ex}\hat{\rho}(x,T)\left\{A(x,T)\right\}^{\frac{\gamma+1}{\gamma-1}}
				\hspace{-1ex}
				\left[\hat{v}(x,T)+\frac{\{\hat{\rho}(x,T)\}^{\theta}}{\theta}\right]dx}
		}{\displaystyle{\left\{\bar{A}_j(T)\right\}^{-1}\frac1{2{\varDelta}x}\int_{x_{j-1}}^{x_{j+1}}\hat{\rho}(x,T)\left\{A(x,T)\right\}^{\frac{\gamma+1}{\gamma-1}}dx}}\hspace{-0.5ex}
		\\&+o({\varDelta}x)
		\\
		\leq&\bar{A}_j(T)+o({\varDelta}x).
	\end{align*}

	Therefore we must prove (\ref{lemma3.1-13}).
	Separating three steps, we derive this estimate.

	{\it Step 2.}\\
	From \eqref{assumption-average2},
	we notice that 
	\begin{align*}
		|\omega|\leq{C}({\varDelta}x)^{-\theta\delta-\varepsilon},
	\end{align*}
	where $C$ depends only on $M$.

	In this step, we consider the first equation of (\ref{lemma3.1-5}):
	\begin{align*}{\left(\frac1{2{\varDelta}x}\int_{x_{j-1}}
			^{x_{j+1}}\hspace{-0.6ex}\hat{\rho}(x,T)
			\left\{A(x,T)\right\}^{\frac{2}{\gamma-1}}
			dx\right)^{\theta+1}}.
	\end{align*}
	Since $\theta\delta<1/2$, we first find 
	\begin{align*}
		E_j^{n+1}(\rho^{\varDelta}_{\dagger})
		=&\frac1{2{\varDelta}x}\int_{x_{j-1}}
		^{x_{j+1}}\hat{\rho}(x,T)\left\{A(x,T)\right\}^{\omega+\frac{2}{\gamma-1}}
		\left\{A(x,T)\right\}^{-\omega}dx\\
		=&\left\{\bar{A}_j(T)\right\}^{-\omega}
		\frac1{2{\varDelta}x}\int_{x_{j-1}}
		^{x_{j+1}}\hat{\rho}(x,T)\left\{A(x,T)\right\}^{\omega+\frac{2}{\gamma-1}}dx\\
		&-\omega
		\left\{\bar{A}_j(T)\right\}^{-\omega-1}\frac1{2{\varDelta}x}\int_{x_{j-1}}
		^{x_{j+1}}\hat{\rho}(x,T)
		\left\{A(x,T)\right\}^{\omega+\frac{2}{\gamma-1}}r(x,T)
		dx\\
		&+o({\varDelta}x)\frac1{2{\varDelta}x}\int_{x_{j-1}}
		^{x_{j+1}}\hat{\rho}(x,T)
		\left\{A(x,T)\right\}^{\frac{2}{\gamma-1}}dx\\
		:=&I_0-I_1+I_2.
	\end{align*}

	We next estimate $I_1$ as follows:
	\begin{align*}
		I_1&=\omega
		\left\{\bar{A}_j(T)\right\}^{-1}\frac1{2{\varDelta}x}\int_{x_{j-1}}
		^{x_{j+1}}\hat{\rho}(x,T)
		\left\{A(x,T)\right\}^{\frac{2}{\gamma-1}}r(x,T)dx
		\nonumber\\
		&\quad+o({\varDelta}x)\frac1{2{\varDelta}x}\int_{x_{j-1}}
		^{x_{j+1}}\hat{\rho}(x,T)\left\{A(x,T)\right\}^{\frac{2}{\gamma-1}}
		dx.
	\end{align*}
	Therefore, we have
	\begin{align*}
		\begin{split}
			E_j^{n+1}(\rho^{\varDelta}_{\dagger})
			=&\frac1{2{\varDelta}x}\int_{x_{j-1}}
			^{x_{j+1}}\hat{\rho}(x,T)\left\{A(x,T)\right\}^{\omega+\frac{2}{\gamma-1}}
			\left\{A(x,T)\right\}^{-\omega}dx\\
			=&\left\{\bar{A}_j(T)\right\}^{-\omega}
			\frac1{2{\varDelta}x}\int_{x_{j-1}}
			^{x_{j+1}}\hat{\rho}(x,T)\left\{A(x,T)\right\}^{\omega+\frac{2}{\gamma-1}}dx\\
			&-\omega
			\left\{\bar{A}_j(T)\right\}^{-1}\frac1{2{\varDelta}x}\int_{x_{j-1}}
			^{x_{j+1}}\hat{\rho}(x,T)
			\left\{A(x,T)\right\}^{\frac{2}{\gamma-1}}r(x,T)dx\\
			&+o({\varDelta}x)\frac1{2{\varDelta}x}\int_{x_{j-1}}
			^{x_{j+1}}\hat{\rho}(x,T)
			\left\{A(x,T)\right\}^{\frac{2}{\gamma-1}}dx.	
		\end{split}
	\end{align*}
	
	From the above, 
	we deduce that 
	\begin{align}
		(E_j^{n+1}(\rho^{\varDelta}_{\dagger}))^{\theta+1}
		&=\left(\left\{\bar{A}_j(T)\right\}^{-\omega}
		\frac1{2{\varDelta}x}\int_{x_{j-1}}
		^{x_{j+1}}\hat{\rho}(x,T)\left\{A(x,T)\right\}^{\omega+\frac{2}{\gamma-1}}dx\right.\nonumber\\
		&\quad\left.
		-\omega
		\left\{\bar{A}_j(T)\right\}^{-1}\frac1{2{\varDelta}x}\int_{x_{j-1}}
		^{x_{j+1}}\hat{\rho}(x,T)
		\left\{A(x,T)\right\}^{\frac{2}{\gamma-1}}r(x,T)dx
		\right)^{\theta+1}\nonumber\\
		&\quad+o({\varDelta}x)\frac1{2{\varDelta}x}\int_{x_{j-1}}
		^{x_{j+1}}\hat{\rho}(x,T)
		\left\{A(x,T)\right\}^{\frac{2}{\gamma-1}}dx
		\nonumber
	\end{align}\begin{align}
		&=\left(
		\left\{\bar{A}_j(T)\right\}^{-\omega}
		\frac1{2{\varDelta}x}\int_{x_{j-1}}
		^{x_{j+1}}\hat{\rho}(x,T)\left\{A(x,T)\right\}^{\omega+\frac{2}{\gamma-1}}dx
		\right)
		^{\theta+1}\nonumber\\
		&\quad+(\theta+1)\left(
		\left\{\bar{A}_j(T)\right\}^{-\omega}
		\frac1{2{\varDelta}x}\int_{x_{j-1}}
		^{x_{j+1}}\hat{\rho}(x,T)\left\{A(x,T)\right\}^{\omega+\frac{2}{\gamma-1}}dx       \right)^{\theta}\nonumber\\
		&\quad\times-\omega
		\left\{\bar{A}_j(T)\right\}^{-1}\frac1{2{\varDelta}x}\int_{x_{j-1}}
		^{x_{j+1}}\hat{\rho}(x,T)
		\left\{A(x,T)\right\}^{\frac{2}{\gamma-1}}r(x,T)dx
		\nonumber\\
		&\quad+o({\varDelta}x)\frac1{2{\varDelta}x}\int_{x_{j-1}}
		^{x_{j+1}}\hat{\rho}(x,T)
		\left\{A(x,T)\right\}^{\frac{2}{\gamma-1}}dx
		\nonumber
		\\
		&=\left(
		\left\{\bar{A}_j(T)\right\}^{-\omega}
		\frac1{2{\varDelta}x}\int_{x_{j-1}}
		^{x_{j+1}}\hat{\rho}(x,T)\left\{A(x,T)\right\}^{\omega+\frac{2}{\gamma-1}}dx
		\right)
		^{\theta+1}\nonumber\\
		&\quad-\frac{\gamma+1}{2}\omega
		\left\{\bar{A}_j(T)\right\}^{-1}\left(
		\frac1{2{\varDelta}x}\int_{x_{j-1}}
		^{x_{j+1}}\hat{\rho}(x,T)\left\{A(x,T)\right\}^{\frac{2}{\gamma-1}}dx       \right)^{\theta}\nonumber\\
		&\quad\times\frac1{2{\varDelta}x}\int_{x_{j-1}}
		^{x_{j+1}}\hat{\rho}(x,T)
		\left\{A(x,T)\right\}^{\frac{2}{\gamma-1}}r(x,T)dx
		\nonumber\\
		&\quad+o({\varDelta}x)\frac1{2{\varDelta}x}\int_{x_{j-1}}
		^{x_{j+1}}\hat{\rho}(x,T)
		\left\{A(x,T)\right\}^{\frac{2}{\gamma-1}}dx
		.
		\label{lemma3.1-9}
	\end{align}

	{\it Step 3}\\
	Applying the Jensen inequality to the first term of the right-hand of (\ref{lemma3.1-9}), we have
	\begin{align}
		&\left(
		\left\{\bar{A}_j(T)\right\}^{-\omega}
		\frac1{2{\varDelta}x}\int_{x_{j-1}}
		^{x_{j+1}}\hat{\rho}(x,T)\left\{A(x,T)\right\}^{\omega+\frac{2}{\gamma-1}}dx
		\right)
		^{\theta+1}\nonumber
		\\
		&=\left(\frac{\displaystyle{	\left\{\bar{A}_j(T)\right\}^{-\omega}
				\frac1{2{\varDelta}x}\int_{x_{j-1}}
				^{x_{j+1}}\hat{\rho}(x,T)\left\{A(x,T)\right\}^
				{\omega+\frac{2}{\gamma-1}}
				dx}}
		{\displaystyle{\frac1{2{\varDelta}x}\int_{x_{j-1}}
				^{x_{j+1}}\left\{A(x,T)\right\}^{\frac{\gamma+1}{\gamma-1}\omega}dx}}\right)^{\theta+1}
		\nonumber
		\\
		&\quad\times\left(\frac1{2{\varDelta}x}\int_{x_{j-1}}
		^{x_{j+1}}\left\{A(x,T)\right\}^{\frac{\gamma+1}{\gamma-1}\omega}dx\right)^{\theta+1}
		\nonumber
		\\
		&=\left(\frac{\displaystyle{	\left\{\bar{A}_j(T)\right\}^{-\omega}
				\frac1{2{\varDelta}x}\int_{x_{j-1}}
				^{x_{j+1}}\hat{\rho}(x,T)\left\{A(x,T)\right\}^
				{\omega+\frac{2}{\gamma-1}}
				dx}}
		{\displaystyle{\frac1{2{\varDelta}x}\int_{x_{j-1}}
				^{x_{j+1}}\left\{A(x,T)\right\}^{\frac{\gamma+1}{\gamma-1}\omega}dx}}\right)^{\theta+1}
		\nonumber\\
		&\quad\times\left(\frac1{2{\varDelta}x}\int_{x_{j-1}}
		^{x_{j+1}}\left\{A(x,T)\right\}^{\frac{\gamma+1}{\gamma-1}\omega}dx\right)
		\nonumber\\
		&\quad\times\left(\left\{\bar{A}_j(T)\right\}^{\frac{\gamma+1}{2}\omega}+
		\frac{\gamma+1}{2}\omega\left\{\bar{A}_j(T)\right\}^{\frac{\gamma+1}{2}\omega-1}\frac1{2{\varDelta}x}\int_{x_{j-1}}
		^{x_{j+1}}r(x,T)dx+o({\varDelta}x)\right)
		\nonumber
	\end{align}\begin{align}
		&=\left(\frac{\displaystyle{	\left\{\bar{A}_j(T)\right\}^{-\omega}
				\frac1{2{\varDelta}x}\int_{x_{j-1}}
				^{x_{j+1}}\hat{\rho}(x,T)\left\{A(x,T)\right\}^
				{\omega-\frac{\gamma+1}{\gamma-1}\omega+\frac{2}{\gamma-1}}
				\left\{A(x,T)\right\}^{\frac{\gamma+1}{\gamma-1}\omega}dx}}
		{\displaystyle{\frac1{2{\varDelta}x}\int_{x_{j-1}}
				^{x_{j+1}}\left\{A(x,T)\right\}^{\frac{\gamma+1}{\gamma-1}\omega}dx}}\right)^{\theta+1}
		\nonumber\\
		&\quad\times\left(\frac1{2{\varDelta}x}\int_{x_{j-1}}
		^{x_{j+1}}\left\{A(x,T)\right\}^{\frac{\gamma+1}{\gamma-1}\omega}dx\right)\left\{\bar{A}_j(T)\right\}^{\frac{\gamma+1}{2}\omega}
		\nonumber\\
		&\quad
		+o({\varDelta}x){\displaystyle{	
				\frac1{2{\varDelta}x}\int_{x_{j-1}}
				^{x_{j+1}}\hat{\rho}(x,T)\left\{A(x,T)\right\}^
				{\frac{2}{\gamma-1}}
				dx}}
		\nonumber
		\\
		&\leq
		\frac1{2{\varDelta}x}\int_{x_{j-1}}
		^{x_{j+1}}(\hat{\rho}(x,T))^{\theta+1}
		\left\{A(x,T)\right\}^{\frac{\gamma+1}{\gamma-1}}dx\nonumber\\&\quad
		+o({\varDelta}x){\displaystyle{	
				\frac1{2{\varDelta}x}\int_{x_{j-1}}
				^{x_{j+1}}\hat{\rho}(x,T)\left\{A(x,T)\right\}^
				{\frac{2}{\gamma-1}}
				dx}}.
		\label{lemma3.1-10}
	\end{align}

	From (\ref{lemma3.1-9}) and (\ref{lemma3.1-10}), we obtain 
	(\ref{lemma3.1-13}) and complete the proof of lemma \ref{lem:average1}.
\end{proof}

\vspace*{-2.0ex}
\section{Construction and $L^{\infty}$ estimates of approximate solutions near the vacuum in Case 1}\label{app:vacuum}

In this step, we consider the case where $\rho_{\rm M}\leq({\varDelta}x)^{\beta}$,
which means that $u_{\rm M}$ is near the vacuum. Since we cannot use the implicit 
function theorem, we must construct $u^{\varDelta}(x,t)$ in a different way.


\vspace*{5pt}
{\bf Case 1} A 1-rarefaction wave and a 2-shock arise.

In this case, we notice that $\rho_{\rm R}\leq ({\varDelta}x)^{\beta},\;
z_{\rm R}\geq -M_n-L_n+I^n_j$ and $w_{\rm R}\leq  M_n+L_n+I^n_j$.
\vspace*{5pt}

\vspace*{5pt}
{\bf Case 1.1}
$\rho_{\rm L}>({\varDelta}x)^{\beta}$

We denote $u^{(1)}_{\rm L}$ a state satisfying $ w(u_{\rm L}^{(1)})=w(u_{\rm L})$ and 
$\rho^{(1)}_{\rm L}=({\varDelta}x)^{\beta}$. 
Let $u^{(2)}_{\rm L}$ be a state connected to ${u}^{\varDelta}_1(x_{j-1},t_{n+1-})$ on the right by 
$R_1^{\varDelta}(u_{\rm L},z^{(1)}_{\rm L},x,t_{n+1-})$. We set 
\begin{align*}
	(z^{(3)}_{\rm L},w^{(3)}_{\rm L})=
	\begin{cases}
		(z^{(2)}_{\rm L},w^{(2)}_{\rm L}),\quad\text{if $z^{(2)}_{\rm L}\geq D^n_j$},\\
		(D^n_j,w^{(2)}_{\rm L}),\quad\text{if $z^{(2)}_{\rm L}< D^n_j$},
	\end{cases}
\end{align*}
where 
\begin{align*}
	D^n_j=&-M_{n+1}-L_n+\int^{x_{j-1}}_{-\infty}J({u}^{\varDelta}_{n,0}(x))dx+V(u_{\rm L}){\varDelta}t+
	\int^{x_{j+1}}_{x_{j-1}}\dfrac{\left(\bar{\rho}\right)^{\gamma}}{\gamma}dx\\
	&+
	\int^{x_{j}+\lambda_1(u^{(2)}_{\rm L}){\varDelta}t}_{x_{j-1}}\eta(R_1^{\varDelta}(u_{\rm L},z^{(1)}_{\rm L},x,t_{n+1-}))dx.
\end{align*}



Then, we define ${u}^{\varDelta}(x,t)$ as follows.
\begin{align*}{u}^{\varDelta}(x,t)=
	\begin{cases}
		R_1^{\varDelta}(u_{\rm L},z^{(1)}_{\rm L},x,t),\hspace*{2ex}\text{if $x_{j-1}
			\leq{x}\leq x_{j}+\lambda_1(u^{(2)}_{\rm L})(t-{t}_{n})$}\\
		\hspace*{19ex}\text{ and ${t}_{n}\leq{t}<{t}_{n+1}$,}\vspace*{1ex}\\
		u_{\rm Rw}(x,t),\hspace*{2ex}\text{if $ x_{j}+\lambda_1(u^{(2)}_{\rm L})(t-{t}_{n})$$<x
			\leq x_{j}+\lambda_2(u_{\rm M},u_{\rm R})(t-{t}_{n})$}\\\hspace*{11ex}\text{ and ${t}_{n}\leq{t}<{t}_{n+1}$,}\vspace*{1ex}\\
		{u}^{\varDelta}_{\rm R}(x,t) \text{ defined in \eqref{appr-R}},\hspace*{2ex}\text{if $x_{j}+\lambda_2(u_{\rm M},u_{\rm R})(t-{t}_{n})$$<x
			\leq x_{j+1}$ }\\\hspace*{28ex}\text{and ${t}_{n}\leq{t}<{t}_{n+1}$,}
	\end{cases}
\end{align*}
where (a) $\lambda_2(u_{\rm M},u_{\rm R})$ is a propagation speed of 2-shock wave; 
(b) $u_{\rm Rw}(x,t)$ is a\linebreak rarefaction wave connecting $u^{(3)}_{\rm L}$ and $u^{(4)}_{\rm L}$; (c) $u^{(4)}_{\rm L}$ is defined by $z^{(4)}_{\rm L}=\max\{z^{(3)}_{\rm L},z_{\rm M}\},$\linebreak$w^{(4)}_{\rm L}=w^{(3)}_{\rm L}$.

\begin{figure}[htbp]
	\begin{center}
		\vspace{-1ex}
		\hspace{2ex}
		\includegraphics[scale=0.3]{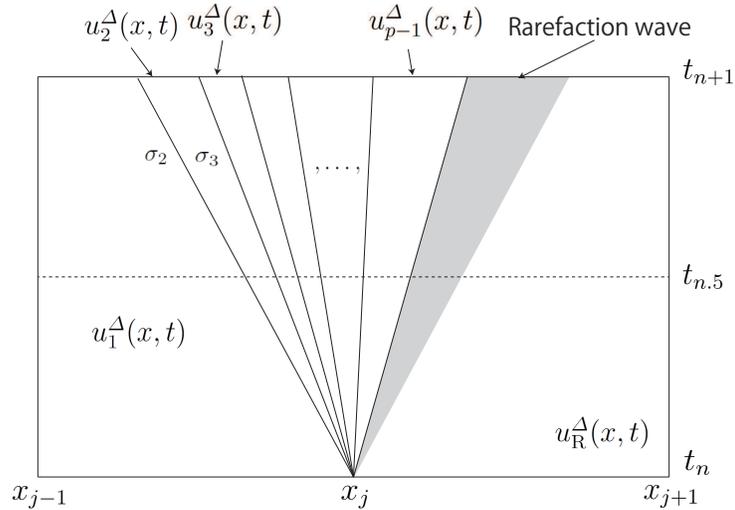}
	\end{center}\vspace*{-2ex}
	\caption{{\bf Case 1.1}: The approximate solution ${u}^{\varDelta}$ in the cell.}
	\label{Fig:case1.1}
\end{figure}

\vspace*{2ex}
{\bf Case 1.2} $\rho_{\rm L}\leq({\varDelta}x)^{\beta}$

We set $(z^{(5)}_{\rm L},w^{(5)}_{\rm L})=(\max\{z_{\rm L},D^n_j\},
\min\{w_{\rm L},U^n_j\})$,
where 
\begin{align*}
	U^n_j=&M_{n+1}+L_n+\int^{x_{j-1}}_{-\infty}J({u}^{\varDelta}_{n,0}(x))dx+V(u_{\rm L}){\varDelta}t.
\end{align*}

Then, we define ${u}^{\varDelta}(x,t)$ as follows.
\begin{align*}{u}^{\varDelta}(x,t)=
	\begin{cases}
		{u}^{\varDelta}_1(x,t)\text{ defined in \eqref{appro1}},\hspace*{2ex}\text{if $x_{j-1}
			\leq{x}\leq x_{j}+\lambda_1(u_{\rm L})(t-{t}_{n})$}\\
		\hspace*{28ex}\text{ and ${t}_{n}\leq{t}<{t}_{n+1}$,}\vspace*{1ex}\\
		u_{\rm Rw}(x,t),\hspace*{2ex}\text{if $ x_{j}+\lambda_1(u_{\rm L})(t-{t}_{n})$$<x
			\leq x_{j}+\lambda_2(u_{\rm M},u_{\rm R})(t-{t}_{n})$}\\\hspace*{11ex}\text{ and ${t}_{n}\leq{t}<{t}_{n+1}$,}\vspace*{1ex}\\
		{u}^{\varDelta}_{\rm R}(x,t) \text{ defined in \eqref{appr-R}},\hspace*{2ex}\text{if $x_{j}+\lambda_2(u_{\rm M},u_{\rm R})(t-{t}_{n})$$<x
			\leq x_{j+1}$ }\\\hspace*{28ex}\text{and ${t}_{n}\leq{t}<{t}_{n+1}$,}
	\end{cases}
\end{align*}
where (a) $u_{\rm Rw}(x,t)$ is a rarefaction wave connecting $u^{(5)}_{\rm L}$ and $u^{(6)}_{\rm L}$; (b) $u^{(6)}_{\rm L}$ is defined by $z^{(6)}_{\rm L}=\max\{z^{(5)}_{\rm L},z_{\rm M}\},\;
w^{(6)}_{\rm L}=w^{(5)}_{\rm L}$.

\begin{remark} 
	We notice that	
	${\rho}^{\varDelta}(x,t)=O(({\varDelta}x)^{\beta})$ in (1.ii), (1.iii) and (2.i)--(2.iii). 
	Therefore, the followings hold in these areas.

	Although (1.ii) and (2.ii) are solutions of homogeneous isentropic gas dynamics (i.e., 
	$g(x,t,u))=0$), they is also a solution of \eqref{IP} approximately
	\begin{align*}
		(u^{\varDelta})_t+f(u^{\varDelta})_x-g(x,u^{\varDelta})=-g(x,u^{\varDelta})=O(({\varDelta}x)^{\beta}).
	\end{align*}

	In addition, discontinuities separating (1.i)--(1.iii) and (2.i)--(2.iii) satisfy \cite[Lemma 5.3]{T1}. 
	
\end{remark}

\subsection{$L^{\infty}$ estimates of approximate solutions}

We consider {\bf Case 1.1} in particular. It suffices to treat with $u_{\rm Rw}(x,t)$ 
in the region where $ x_{j}+\lambda_1(u^{(2)}_{\rm L})(t-{t}_{n})<x
\leq x_{j}+\lambda_2(u_{\rm M},u_{\rm R})(t-{t}_{n})$ and ${t}_{n}\leq{t}<{t}_{n+1}$. 
The other cases are similar to Theorem \ref{thm:bound}.

In this case, since ${\rho}^{\varDelta}(x,t)=O(({\varDelta}x)^{\beta})$, we have 
\begin{align}
	\eta_{\ast}({u}^{\varDelta}(x,t))=O(({\varDelta}x)^{\beta}).
	\label{vacuum-eta}
\end{align}
Moreover, we notice that \begin{align*}{w}^{\varDelta}(x,t_{n+1-})=w^{(2)}_{\rm L}=w(R_1^{\varDelta}(u_{\rm L},z^{(1)}_{\rm L},x_{j}+\lambda_1(u^{(2)}_{\rm L}){\varDelta}t,t_{n+1-})).\end{align*}
Applying Theorem \ref{thm:bound} to $R_1^{\varDelta}(u_{\rm L},z^{(1)}_{\rm L},x,t_{n+1-})$, we drive 
\begin{align*}
	\begin{alignedat}{2}
		&\displaystyle {w}^{\varDelta}(x,t_{n+1-})
		&\leq& M_{n+1}+L_n+\int^{x_{j}+\lambda_1(u^{(2)}_{\rm L}){\varDelta}t}_{-\infty}J({u}^{\varDelta}(y,t_{n+1-}))dy\\&&&+\int^{t_{n+1}}_{t_n}\sum_{y<x_{j-1}}(\sigma[\eta_{\ast}]-[q_{\ast}])dt+{\it o}({\varDelta}x)\\
		&&\leq& M_{n+1}+L_n+\int^{x}_{-\infty}J({u}^{\varDelta}(y,t_{n+1-}))dy+\int^{t_{n+1}}_{t_n}\sum_{y<x_{j-1}}(\sigma[\eta_{\ast}]-[q_{\ast}])dt\\&&&+{\it o}({\varDelta}x),
	\end{alignedat}
\end{align*}
which means $\eqref{goal}_2$.

Next, we notice that ${z}^{\varDelta}(x,t)\geq D^n_j$. In view of \eqref{mass-conservation} and \eqref{vacuum-eta}, we obtain $\eqref{goal}_1$.


\section*{Acknowledgements.}
N. Tsuge's research is partially supported by Grant-in-Aid for Scientific Research (C) 17K05315, Japan. 
  
\end{document}